\documentclass[12pt,a4paper, reqno]{amsart}
\usepackage{}
\usepackage{txfonts}
\usepackage{color}
\usepackage{amsmath,amssymb,extarrows}
\usepackage{tikz,enumerate}
\usepackage{diagbox}
\usepackage{appendix}
\usepackage{epic}
\usepackage{fourier}
\usepackage{float}
\usepackage{bbm}
\usepackage{amsfonts}
\usepackage{amsxtra}
\usepackage{amssymb}
\usepackage{amscd}
\usepackage{amsthm}
\usepackage{mathrsfs}
\usepackage{hyperref}
\usepackage{leqno}
\usepackage{multirow}
\usepackage{tikz}
\usepackage{tikz-cd}

\makeatletter

\numberwithin{equation}{section}
\newtheorem*{theorem*}{Theorem}
\newtheorem{lemma}{Lemma}[section]

\newtheorem{remark}[lemma]{Remark}
\newtheorem{example}[lemma]{Example}

\newtheorem{definition}[lemma]{Definition}

\newtheorem{corollary}[lemma]{Corollary}

\newtheorem*{question*}{Question}
\newtheorem*{assumption*}{Assumption}
\newtheorem*{axiom*}{Axiom}

\newtheorem*{theorem*1}{Theorem (\ref{theta1})}
\newtheorem*{theorem*2}{Theorem (\ref{theta2})}
\newtheorem*{theorem*3}{Theorem (\ref{theta3})}
\newtheorem*{theorem*4}{Theorem (\ref{theta4})}
\newtheorem*{proposition*5}{Proposition (\ref{theta5})}
\newtheorem*{proposition*6}{Proposition (\ref{theta6})}
\baselineskip 20pt \textwidth 18cm  \theoremstyle{plain}
\usepackage[all,cmtip]{xy}

\newcommand{\cInd}{\operatorname{c-Ind}}

\newcommand{\Ha}{\operatorname{H}}

\newcommand{\C}{\mathbb C}

\newcommand{\Z}{\mathbb Z}

\newcommand{\GL}{\operatorname{GL}}

\newcommand{\GSp}{\operatorname{GSp}}
\newcommand{\GMp}{\operatorname{GMp}}

\newcommand{\PGSp}{\operatorname{PGSp}}
\newcommand{\Mp}{\operatorname{Mp}}
\newcommand{\PGMp}{\operatorname{PGMp}}

\newcommand{\SL}{\operatorname{SL}}

\newcommand{\Sp}{\operatorname{Sp}}

\newcommand{\Span}{\operatorname{Span}}

\begin{document}
\title{Extended  Weil representations: the non-dyadic local field cases}
\author{Chun-Hui Wang}
\address{School of Mathematics and Statistics\\Wuhan University \\Wuhan, 430072,
P.R. CHINA}
\email{cwang2014@whu.edu.cn}
\begin{abstract} Let  $F$ be  a non-archimedean local field of odd residual characteristic. Let $W$ be a symplectic vector space over $F$. It is  known that  there are   different Weil representations of a Meteplectic covering group $\Mp(W)$. By utilizing  twisted actions, we  reorganize these representations as a representation of $\PGMp^{\pm}(W)$, a covering group related to the projective similitude symplectic group.  
\end{abstract}
\maketitle
\setcounter{secnumdepth}{3}
\tableofcontents{}
\section{Introduction to the main result}\label{In}
 Let  $F$ be  a non-archimedean local field of odd residual characteristic. Let $(W, \langle, \rangle)$ be a symplectic vector space  of  dimension $2m$   over  $F$.  Let $\Ha(W)=W\oplus F$ denote the usual Heisenberg group.  Denote the corresponding symplectic group (resp. similitude symplectic group) as $\Sp(W)$(resp. $\GSp(W)$).  Let $\psi$ be a non-trivial character of $F$.  Consider the 8-fold metaplectic group $\Mp(W)$ over $\Sp(W)$. Let $\pi_{\psi}$ be the  Weil representation of $\Mp(W)$ associated to $\psi$. For more information  on  Weil representations in the $p$-adic field case, refer to  \cite{Ku}, \cite{MoViWa}, etc. One question that arises is how to extend  Weil representations to the similitude groups.   In the $p$-adic case, this question has been systematically studied by Barthel  in \cite{Ba}.  In  that paper,  Barthel constructed the Metaplectic $2$-cocycle over $\GSp(W)$, defined the corresponding   Metaplectic group $\GMp(W)$, and then  obtained the  Weil representation  $\cInd_{\Mp(W)}^{\GMp(W)} \pi_{\psi}$. Inspired by \cite{GePi},\cite{Wa1},\cite{Wa2}, in this paper we will continue their works by narrowing the Weil representation through some twisted actions.

 Let $ a$ be an element of $ F^{\times}$ and  let $\psi^a$ denote another character of $F$ defined as $t\longmapsto \psi(at)$, for $t\in F$. Let $\pi_{\psi^a}$ be the Weil representation of $\Mp(W)$ associated to $\psi^a$. According to \cite[p.36(4)]{MoViWa}, $\pi_{\psi^{at^{2}}}\simeq \pi_{\psi^a}$, for any $t\in F^{\times}$. Let $k_F$ be the residue field of $F$. To simply the induction, we will first  assume that the cardinality of $|k_F|$ is congruent to $3$ modulo $4$. Our purpose is to find a group $\PGSp^{\pm}(W)$ such that there exists an exact sequence:
\begin{align*}
1 \longrightarrow \Sp(W)\longrightarrow \PGSp^{\pm}(W)\longrightarrow F^{\times}/F^{\times 2}\longrightarrow 1.
\end{align*}
 Then extend the Metaplectic $2$-cocycle from $\Sp(W)$ to $\PGSp^{\pm}(W)$.

  One issue arises because the center group of the  Metaplectic group constructed by Barthel  in \cite{Ba} is not the covering group over $F^{\times}$ when  the   dimension of $W$ does not divide $4$. So in Section \ref{sectionmap}, 
 we select  another  section map $\iota$ from $F^{\times}$ to $\GSp(W)$ so as to construct  a group $\widetilde{F^{\times}}$, which is a central extension of $F^{\times}/F^{\times 2}$ by $F^{\times}$. We define:
 $$\PGSp^{\pm}(W)= [\widetilde{F^{\times}} \ltimes \Sp(W)]/[F^{\times} \times 1].$$
Then there exists the above  exact sequence.  Based on the lemma 1.1.B in \cite{Ba},  we construct a cocycle $\widetilde{C}_M(-,-)$ over $\widetilde{F^{\times}} \ltimes \Sp(W)$.  Associated with the cocycle $\widetilde{C}_M(-,-)$ is the Metaplectic group $\widetilde{F^{\times}} \ltimes \Mp(W)$, whose central group contains  $[F^{\times} \times 1]$(cf. Lemma \ref{cent3mdulo4}). Then we can  define:
 $$\PGMp^{\pm}(W)=[\widetilde{F^{\times}} \ltimes \Mp(W)]/[F^{\times}\times 1].$$  This leads to two exact sequences: 
 \begin{align*}
1 \longrightarrow \mu_8\longrightarrow \PGMp^{\pm}(W)\longrightarrow \PGSp^{\pm}(W)\longrightarrow 1.
\end{align*}
\begin{align*}
1 \longrightarrow \Mp(W)\longrightarrow \PGMp^{\pm}(W)\longrightarrow F^{\times}/F^{\times 2}\longrightarrow 1.
\end{align*}
As a result, we can define the twisted induced Weil representation of  $\PGMp^{\pm}(W)$: 
$$\Pi=\cInd_{\Mp(W)}^{\PGMp^{\pm}(W)} \pi_{\psi},$$
which consists of  the four different Weil representations of  $\Mp(W)$. We also consider the 2-fold Metaplectic covering group, following the works of P. Perrin, R. Ranga Rao, and S. Kudla (see \cite{Pe},\cite{Ra},\cite{Ku}). However, when attempting to incorporate the Heisenberg group, there are certain issues that differ from the finite field case (see  Section \ref{atwstamod134}, \cite{Wa3}). Therefore, it is uncertain whether this Weil representation is suitable for studying the Howe correspondence.

In  Sections \ref{case1mod411}-\ref{case1mod414}, we  address the case where $|k_F|\equiv 1(\bmod4)$  by considering a covering group  instead of $F^{\times}$ itself. We obtain the same result as above  in the case of the $8$-degree Metaplectic cover. However, when  the   dimension of $W$ does not divide $4$, we are only able to extend the $2$-degree Metaplectic $2$-cocycle on $\Sp(W)$ to a $4$-degree Metaplectic $2$-cocycle on $\PGSp^{\pm}(W)$.(cf. Lemmas \ref{CM111},\ref{CM112}) 

For recent developments on the twisted problems and conjectures, see the paper by Gan, Gross, and Prasad  \cite{GaGrPr}. One question that arises is whether groups $\PGMp^{\pm}(W)$ also possess the GGP property.
 \section{Preliminaries}
 \subsection{Notations and conventions}\label{Notationandconventionsis}
In the whole text, we will  use the following notion and conventions. We will let  $F$ be  a non-archimedean local field  of \emph{odd} residual characteristic.  Let $k_F$  denote  its residue field. Let $\varpiup$ be one uniformizer of $F$.  Let  $\mathfrak{O}$ denote the ring of integers, and $\mathfrak{P}$ its prime ideal.  Let $U$ be the group of units in $\mathfrak{O}$, and  $U_n=\{ u \in F^{\times} \mid   u\equiv 1 \mod \mathfrak{P}^{n} \}$. Then $U/{U_1} \simeq k_F^{\times}$, and $U\simeq U_1 \times \mathfrak{f}$, for certain cyclic  subgroup $\mathfrak{f}$ of $U$. Assume $\mathfrak{f}=\langle \zeta \rangle$.

Let $(W, \langle, \rangle)$ be a symplectic vector space  of  dimension $2m$   over  $F$.  Let $\Ha(W)=W\oplus F$ denote the usual Heisenbeg group, with the multiplication law given by
$$(w, t)(w',t')=(w+w', t+t'+\frac{\langle w,w'\rangle}{2}).$$
 Let   $\Sp(W)$(resp.$\GSp(W)$) denote the corresponding symplectic group(resp. similitude symplectic group). Let $\lambda$ be the similitude factor map from $\GSp(W)$ to $F^{\times}$, and for $g\in \GSp(W)$, let $\lambda_g$ denote its similitude factor.  Let $\{e_1, \cdots, e_m; e_1^{\ast}, \cdots, e_m^{\ast}\}$ be a symplectic basis of $W$ so that $\langle e_i, e_j\rangle=0=\langle e_i^{\ast}, e_j^{\ast}\rangle$, $\langle e_i, e_j^{\ast}\rangle =\delta_{ij}$. Let $X=\Span\{ e_1, \cdots, e_m\}$, $X^{\ast}=\Span\{e_1^{\ast}, \cdots, e_m^{\ast}\}$.

 In $\Sp(W)$, let  $P=P(X^{\ast})=\{ g\in \Sp(W) \mid X^{\ast}g=X^{\ast}\}=\{ \begin{pmatrix} a& b\\ 0& (a^{\ast})^{-1}\end{pmatrix}\in \Sp(W)\}$, $N=\{ \begin{pmatrix} 1& b\\ 0& 1\end{pmatrix} \in \Sp(W)\}$, $M=\{ \begin{pmatrix} a& 0\\ 0& (a^{\ast})^{-1}\end{pmatrix}\in \Sp(W)\}$, $N^{-}=\{ \begin{pmatrix} 1& 0\\ c& 1\end{pmatrix} \in \Sp(W)\}$.  For a subset $S\subseteq \{1, \cdots, m\}$, let us define an element   $\omega_S$ of  $ \Sp(W)$ as follows: $ (e_i)\omega_S=\left\{\begin{array}{lr}
-e_i^{\ast}& i\in S\\
 e_i & i\notin S
 \end{array}\right.$ and $ (e_i^{\ast})\omega_S=\left\{\begin{array}{lr}
e_i^{\ast}& i\notin S\\
 e_i & i\in S
 \end{array}\right.$.  In particular,  we will  write simply  $\omega$,  for  $S=\{1, \cdots, m\}$.

  Let $(-, -)_F$ denote the Hilbert symbol defined from $F^{\times} \times F^{\times}$ to $\{ \pm 1\}$. If  $(Q, V)$ is a quadratic form defined over $F$ with the Witt decomposition $V\simeq \oplus_{i=1}^l F(a_i)$, the Hasse invariant is given in the following form:
$\epsilon(Q):=\prod_{1 \leq i< j\leq l} (a_i, a_j)_F$. We will let $\mu_n=\langle e^{\frac{2\pi i}{n}}\rangle$,  $ e^{\frac{2\pi i}{n}} \in \C^{\times}$.
\subsection{$2$-cocycles}
By Stone-von Neumann's theorem,  there exists only one irreducible representation of $\Ha(W)$ with central character $\psi$, up to isomorphism.  This representation is known as the Heisenberg representation and is denoted by $\pi_{\psi}$.  The symplectic group $\Sp(W)$ can act on $\Ha(W)$ while leaving the center $F$ unchanged. According to Weil's famous work, $\pi_{\psi}$ can give rise to a projective representation of $\Sp(W)$, which in turn leads to a true representation of a $\C^{\times}$-covering group over $\Sp(W)$. 
This group is commonly referred to as the Metaplectic group, and the true representation is known as the Weil representation. This special central covering can be reduced to an $8$-degree or $2$-degree covering over $\Sp(W)$, as shown in works by  \cite{Pe}, \cite{Ra}, \cite{We}, and others. The corresponding $8$-covering (or $2$-covering) Metaplectic group is denoted as $\Mp(W)$(resp. $\overline{\Sp}(W)$). We will mainly refer to the results of Rao's paper  \cite{Ra} and  Kudla's note \cite{Ku}.
\subsubsection{Perrin-Rao's $2$-cocycle} For $\psi$,  let $\gamma$ denote the  corresponding Weil index.  Note that the Weil index only depends on  the Witt class  and $\psi$.  For simplicity, we will denote  by  $\gamma_{\psi}(a)$--- the Weil index attached to the quadratic form  $v \longmapsto av^2$, and  let $\gamma(a, \psi)= \tfrac{\gamma_{\psi}(a)}{\gamma_{\psi}(1)}$ be its normalizer.

Let $c_{PR, X^{\ast}}(-,-)$ be the Perrin-Rao's cocycle in $\Ha^2(\Sp(W), \mu_8)$ with respect to $W=X\oplus X^{\ast}$ and $\psi$. For $g_1, g_2 \in \Sp(W)$, let $q(g_1,g_2)=q(X^{\ast}, X^{\ast} g_2^{-1}, X^{\ast} g_1)$ be the corresponding Leray invariant.  Then:
\begin{align}
c_{PR, X^{\ast}}(g_1, g_2)=\gamma(\psi(q(g_1, g_2)/2)).
\end{align}
As is known that  $\Sp(W)= \sqcup_{j=0}^m C_j$,
where $C_j= P\omega_SP$ for the above  $\omega_S$ with $|S|=j$.   In \cite{Ra}, Rao defined the following functions:
\begin{align}
x: \Sp(W) \longrightarrow F^{\times}/{(F^{\times})}^2; p_1\omega_Sp_2 \longmapsto \det(p_1p_2|_{X^{\ast}}) (F^{\times })^2,
\end{align}
\begin{align}
t: \Sp(W) \times \Sp(W) \longrightarrow \Z; (g_1, g_2)\longmapsto \tfrac{1}{2}(|S_1| + |S_2| - |S_3| -l),
\end{align}
 where $ g_1=p_1\omega_{S_1} p_1', g_2= p_2 \omega_{S_2} p_2'$  and  $ g_1g_2= p_3 \omega_{S_3} p_3', l=\dim q(g_1, g_2)$.  The  Perrin-Rao's $2$-cocycle in $\Ha^2(\Sp(W),\{ \pm 1\})$ is defined as follows:
\begin{align}
c(g_1, g_2)= (x(g_1), x(g_2))_F(-x(g_1)x(g_2), x(g_1g_2))_F ((-1)^t, \det(2q))_F (-1, -1)_F^{\tfrac{t(t-1)}{2}}\epsilon(2q),
\end{align}
where $t=t(g_1,g_2)$, $q=q(g_1,g_2)$, for  $g_1, g_2 \in \Sp(W)$. By \cite[p.20, Thm.4.5]{Ku} or \cite[p.360, Def.5.2]{Ra},  we can  define the normalizing constant by
\begin{align}\label{mx}
m_{X^{\ast}}: \Sp(W) \longrightarrow \mu_8; g \longmapsto \left\{ \begin{array}{c} (x(g), \tfrac{1}{2})_F \gamma(x(g), \psi)^{-1} \gamma_{\psi}(\tfrac{1}{2})^{-j(g)}\\
=\gamma(x(g), \psi^{\tfrac{1}{2}})^{-1} \gamma_{\psi^{\tfrac{1}{2}}}(1)^{-j(g)}\end{array}\right.
\end{align}
 for $g=p_1\omega_Sp_2$, $j(g)=|S|$. Then:
\begin{equation}\label{28inter}
c(g_1, g_2)=m_{X^{\ast}}(g_1g_2)^{-1} m_{X^{\ast}}(g_1) m_{X^{\ast}}(g_2) c_{PR, {X^{\ast}}}(g_1,g_2).
\end{equation}
\begin{example}
If $\dim W=2$, $\Sp(W)\simeq \SL_2(F)$. For $g_1, g_2, g_3=g_1g_2\in \SL_2(F)$ with $g_i=\begin{pmatrix}
a_i & b_i\\
c_i& d_i\end{pmatrix}$, we have:
\begin{itemize}
\item[(1)] $x(g_1)=\left\{\begin{array}{lr} a_1F^{\times 2} & \textrm{ if } c_1=0\\
c_1F^{\times 2} & \textrm{ if } c_1\neq 0\end{array}\right.$;
\item[(2)] $m_{X^{\ast}}(g_1)=\left\{\begin{array}{lr} \gamma(a_1, \psi^{\tfrac{1}{2}})^{-1}& \textrm{ if } c_1=0\\
\gamma_{\psi}(\tfrac{1}{2} c_1)^{-1}& \textrm{ if } c_1\neq 0\end{array}\right.$;
\item[(3)] $c_{PR, X^{\ast}}(g_1, g_2)=\gamma_{\psi}(\tfrac{1}{2} c_1c_2c_3)$;
\item[(4)] $c(g_1, g_2)=(x(g_1), x(g_2))_F(-x(g_1)x(g_2), x(g_3))_F$.
\end{itemize}
\end{example}
\begin{proof}
See \cite{Ku}, \cite{Ra}.
\end{proof}
\subsubsection{Barthel's $2$-cocycle I}
In \cite{Ba}, Barthel generated these $2$-cocycles   from $\Sp(W)$ to $\GSp(W)$ and used them to  construct the  Metaplectic covering group over $\GSp(W)$. According to \cite[Lmm.1.1.B]{Ba}, an automorphism  of $\Sp(W)$ can be  uniquely lifted to an automorphism  of  $\Mp(W)$ or $\overline{\Sp}(W)$   by keeping the center $\mu_8$ or $\mu_2$ unchanged.

  Let us first consider the Metaplectic group $\Mp(W)$.  Let $\alpha$ be an automorphism of $\Sp(W)$. Following \cite[p.210]{Ba}, let $\nu$ denote the map from $\alpha \times \Sp(W) $ to $ \mu_8$ such that the following equality holds:
$$(g, \epsilon)^{\alpha}=(g^{\alpha}, \nu(\alpha, g) \epsilon), \qquad  \qquad (g, \epsilon) \in \Mp(W).$$
By \cite[p.210, Sect.1.2.1]{Ba}, the above function $\nu$ determines an automorphism of $\Mp(W)$ iff the following equality holds:
\begin{equation}\label{crao}
c_{PR, X^{\ast}}(g, g')=c_{PR, X^{\ast}}(g^{\alpha},g^{'\alpha})\nu(\alpha,g) \nu( \alpha, g') \nu(\alpha, gg')^{-1}.
\end{equation}
\begin{example}\label{example1}
If $\alpha=h\in  \Sp(W)$, and  $h$ acts on $\Sp(W)$ by  conjugation, then $\nu(\alpha,g)=c_{PR, X^{\ast}}(h^{-1}, g h)c_{PR, X^{\ast}}(g,h)$.
\end{example}
\begin{proof}
\begin{equation}
\begin{split}
&c_{PR, X^{\ast}}(g_1^h, g_2^h) c_{PR, X^{\ast}}(h^{-1}, g_1h)c_{PR, X^{\ast}}(h^{-1}, g_1g_2h)^{-1}\\
=&c_{PR, X^{\ast}}(g_1h, h^{-1}g_2h)\\
=&c_{PR, X^{\ast}}(h^{-1}, g_2h)^{-1}c_{PR, X^{\ast}}(g_1h, h^{-1}) c_{PR, X^{\ast}}(g_1, g_2h)\\
=&c_{PR, X^{\ast}}(h^{-1}, g_2h)^{-1}c_{PR, X^{\ast}}(g_1,h)^{-1} c_{PR, X^{\ast}}(h,h^{-1})c_{PR, X^{\ast}}(g_2,h)^{-1} c_{PR, X^{\ast}}(g_1,g_2)c_{PR, X^{\ast}}(g_1g_2,h)\\
=&c_{PR, X^{\ast}}(g_1,g_2) c_{PR, X^{\ast}}(h^{-1}, g_2h)^{-1}c_{PR, X^{\ast}}(g_1,h)^{-1} c_{PR, X^{\ast}}(g_2,h)^{-1} c_{PR, X^{\ast}}(g_1g_2,h).
\end{split}
\end{equation}
Hence \begin{equation}
\begin{split}
&\frac{c_{PR, X^{\ast}}(g_1^h, g_2^h)}{c_{PR, X^{\ast}}(g_1,g_2)}\\
=&[c_{PR, X^{\ast}}(h^{-1}, g_1g_2h)c_{PR, X^{\ast}}(g_1g_2,h)][c_{PR, X^{\ast}}(h^{-1}, g_1h)c_{PR, X^{\ast}}(g_1,h)]^{-1} [c_{PR, X^{\ast}}(h^{-1}, g_2h)c_{PR, X^{\ast}}(g_2,h)]^{-1}.
\end{split}
\end{equation}

So $\nu(\alpha,g)=c_{PR, X^{\ast}}(h^{-1}, g h)c_{PR, X^{\ast}}(g,h)$ by the uniqueness.
\end{proof}
\begin{example}[Barthel]\label{example2}\footnote{ The expression here differs from Barthel’s formula by $-1$. This is because I have written the action of $\Sp(W)$ on $\Ha(W)$ on the right, which means that the notation $\omega_S$ we use is not the same. }
If $\alpha=\begin{pmatrix}
1& 0\\
0& y
\end{pmatrix}$, for some $y\in F^{\times}$,  then
\begin{align}\label{alpp}
\nu(\alpha,g)=(\det (a_1a_2), y)_F \gamma(y, \psi^{\frac{1}{2}})^{-|S|},
\end{align} for $g=p_1\omega_S p_2$, $p_i=\begin{pmatrix}
a_i& b_i\\
0& d_i
\end{pmatrix}$, $i=1,2$.
\end{example}
\begin{proof}
See \cite[Pro.1.2.A]{Ba}. Since we follow Rao and Kudla's  notations, let us verify this result in more details. By \cite[p.212]{Ba}, there exist a character $\chi_y$ of $F^{\times}$and a non -zero complex number $\beta_y$ such that
$$   \nu(y, \begin{pmatrix}
a& b\\
0& (a^{\ast})^{-1}
\end{pmatrix})= \chi_y(\det a), \textrm{ and }  \nu(y, \omega_S)=\beta_y^{|S|}.$$
By the  embedding  $\SL_2\longrightarrow \Sp$, it reduces to discussing the simple case that  $\dim W=2$.  Following \cite[p.212]{Ba}, let $N(x)= \begin{pmatrix}
1& 0\\
x& 1
\end{pmatrix}\in \SL_2$. Since
$$ \begin{pmatrix}
1& 0\\
x& 1
\end{pmatrix}= \begin{pmatrix}
1& x^{-1}\\
0& 1
\end{pmatrix} \begin{pmatrix}
0& -1\\
1& 0
\end{pmatrix} \begin{pmatrix}
x& 1\\
0& x^{-1}
\end{pmatrix},$$
$\nu(y, N(x))=\beta_y \chi_y(x)$.  For $a, b, a+b \in F^{\times}$,
\begin{equation}
\begin{split}
\beta_y \chi_y(a+b)=\nu(y, N(a+b))&=\nu(y, N(a))\nu(y, N(b))\frac{c_{PR, X^{\ast}}(N(a)^y,N(b)^y)}{c_{PR, X^{\ast}}(N(a),N(b))}\\
& =\nu(y, N(a))\nu(y, N(b))\frac{c_{PR, X^{\ast}}(N(ay^{-1}),N(by^{-1}))}{c_{PR, X^{\ast}}(N(a),N(b))}\\
&=\beta_y \chi_y(a)\beta_y \chi_y(b)\frac{\gamma_{\psi}(\frac{1}{2} ab(a+b)y^{-3})}{\gamma_{\psi}(\frac{1}{2} ab(a+b))}\\
&=\beta_y^2 \chi_y(ab)(ab(a+b), y)_F\gamma(y, \psi^{\frac{1}{2}}).
\end{split}
\end{equation}
Hence:
\begin{align}
\beta_y^{-1} \chi_y(\frac{1}{a}+\frac{1}{b})=(\frac{1}{a}+\frac{1}{b}, y)_F  \gamma(y, \psi^{\frac{1}{2}}).
\end{align}
Therefore, $\chi_y(a)=(a,y)_F$,  $\beta_y^{-1}=\gamma(y, \psi^{\frac{1}{2}})$.
\end{proof}
\begin{lemma}\label{twoau}
For two automorphisms $\alpha_1, \alpha_2$,  $\nu(\alpha_1\alpha_2, g)=\nu(\alpha_1,g)\nu(\alpha_2, g^{\alpha_1})$.
\end{lemma}
\begin{proof}
Because  $(g, \epsilon)^{\alpha_1\alpha_2}=(  g^{\alpha_1}, \nu( \alpha_1,g) \epsilon)^{\alpha_2}=( g^{\alpha_1\alpha_2}, \nu(\alpha_1,g) \nu(\alpha_2,g^{\alpha_1})\epsilon)$.
\end{proof}
Recall the exact sequence:
\begin{align}
1 \longrightarrow \Sp(W) \longrightarrow \GSp(W) \stackrel{\lambda}{\longrightarrow} F^{\times}\longrightarrow 1.
\end{align}
Let us choose a section map
$$s: F^{\times} \longrightarrow  \GSp(W); y \longmapsto \begin{pmatrix}
1& 0\\
0& y
\end{pmatrix}.$$
Then $\GSp(W) \simeq  F^{\times}\ltimes \Sp(W)$. By lifting the action of $F^{\times}$ onto $\Mp(W)$,  we   obtain a group $F^{\times} \ltimes \Mp(W)$ and an exact sequence:
$$ 1\longrightarrow \mu_8 \longrightarrow F^{\times} \ltimes \Mp(W) \stackrel{ }{\longrightarrow}F^{\times} \ltimes \Sp(W) \longrightarrow 1.$$
 Let us recall  Barthel' result in constructing the $2$-cocycle $C_{BPR, X^{\ast}}$ associated to this exact sequence. For  $([y_1, g_1], \epsilon_1), ([y_2, g_2], \epsilon_2)\in F^{\times} \ltimes \Mp(W)$,
\begin{equation}
\begin{split}
(y_1, g_1, \epsilon_1)\cdot (y_2, g_2, \epsilon_2)&=(y_1y_2, [g_1, \epsilon_1]^{y_2}[g_2, \epsilon_2])\\
& =(y_1y_2, [g_1^{y_2}, \nu(y_2, g_1)\epsilon_1][g_2, \epsilon_2])\\
&=(y_1y_2, [g_1^{y_2}g_2, \nu( y_2,g_1)c_{PR, X^{\ast}}(g_1^{y_2}, g_2)\epsilon_1\epsilon_2]).
\end{split}
\end{equation}
Hence:
\begin{align}
C_{BPR, X^{\ast}}([y_1,g_1], [y_2, g_2]) =\nu( y_2,g_1)c_{PR, X^{\ast}}(g_1^{y_2}, g_2),   \quad\quad [y_i, g_i]\in F^{\times} \ltimes \Sp(W).
\end{align}
\begin{example}
If $\dim W=2$, $\GSp(W)\simeq \GL_2(F)$. Let  $h_i=\begin{pmatrix}
a_i& b_i\\
c_i& d_i
\end{pmatrix}\in \GL_2(F)$,  $i=1, 2, 3$, with $h_3=h_1h_2$.  Then:
$$C_{BPR, X^{\ast}}(h_1, h_2) =\left\{ \begin{array}{ll}
 (a_1, \det h_2)_F, & \textrm{ if } c_1=0, \\
(c_1\det h_1, \det h_2)_F \gamma(\det h_2, \psi^{\frac{1}{2}})^{-1}\gamma_{\psi}(\tfrac{1}{2}c_1 c_2c_3\det h_2), & \textrm{ if } c_1\neq 0.
 \end{array}\right. $$
\end{example}
\begin{proof}
Note that for $g= \begin{pmatrix}
a& b\\
c& d
\end{pmatrix} \in \SL_2(F)$, with $c\neq 0$,
$g= \begin{pmatrix}
1& a/c\\
0& 1
\end{pmatrix}\begin{pmatrix}
0& -1\\
1& 0
\end{pmatrix}\begin{pmatrix}
c& d\\
0& c^{-1}
\end{pmatrix}$. Let us  write $h_i=\begin{pmatrix}
1& 0\\
0& y_i
\end{pmatrix} g_i$, for $y_i=\det h_i$, $g_i=\begin{pmatrix}
a_i& b_i\\
y_i^{-1}c_i& y_i^{-1}d_i
\end{pmatrix}$. Then $$C_{BPR, X^{\ast}}(h_1, h_2)=C_{BPR, X^{\ast}}([y_1,g_1], [y_2, g_2])=\nu( y_2,g_1)c_{PR, X^{\ast}}(g_1^{y_2}, g_2).$$
1) In the first case,  $c_{PR, X^{\ast}}(g_1^{y_2}, g_2)=1$. Hence:
$$C_{BPR, X^{\ast}}(h_1, h_2)=\nu( y_2,g_1)=(y_2,  a_1)_F.$$
2) In the second case, if $c_1\neq 0$, $c_2=0$, $c_{PR, X^{\ast}}(g_1^{y_2}, g_2)=1$.
$$\nu( y_2,g_1)=(y_1^{-1} c_1, y_2)_F \gamma(y_2, \psi^{\frac{1}{2}})^{-1}=(y_1 c_1, y_2)_F\gamma(y_2, \psi^{\frac{1}{2}})^{-1}.$$
 Hence: $$C_{BPR, X^{\ast}}(h_1, h_2)=\nu( y_2,g_1)=(y_1 c_1, y_2)_F\gamma(y_2, \psi^{\frac{1}{2}})^{-1}.$$.\\
If $c_1\neq 0$, $c_2\neq 0$, $g_1^{y_2}= \begin{pmatrix}
a_1& b_1 y_2\\
y_2^{-1}y_1^{-1}c_1& y_1^{-1}d_1
\end{pmatrix}$. So
$$ c_{PR, X^{\ast}}(g_1^{y_2}, g_2)= \gamma_{\psi}(\tfrac{1}{2}y_2^{-1}y_1^{-1}c_1 y_2^{-1}c_2(y_2^{-1}y_1^{-1}c_1a_2+y_1^{-1}d_1 y_2^{-1}c_2))$$
$$=\gamma_{\psi}(\tfrac{1}{2}c_1 y_2^{-1}c_2(c_1a_2+d_1c_2))=\gamma_{\psi}(\tfrac{1}{2}y_2c_1c_2c_3);$$
$$\nu( y_2,g_1)=(y_1 c_1, y_2)_F\gamma(y_2, \psi^{\frac{1}{2}})^{-1}. $$
$$C_{BPR, X^{\ast}}(h_1, h_2)=\nu( y_2,g_1)c_{PR, X^{\ast}}(g_1^{y_2}, g_2)=(y_1 c_1, y_2)_F\gamma(y_2, \psi^{\frac{1}{2}})^{-1}\gamma_{\psi}(\tfrac{1}{2}y_2c_1c_2c_3).$$
\end{proof}
\subsubsection{Barthel's $2$-cocycle II}
As a consequence, let us consider the $2$-covering Metaplectic group $\overline{\Sp}(W)$. For an automorphism $\alpha$ of $\Sp(W)$, let us define $\nu_2: \alpha \times \Sp(W)  \longrightarrow \mu_2$ such that
$$(g, \epsilon)^{\alpha}=(g^{\alpha}, \nu_2( \alpha,g) \epsilon), \qquad  \qquad (g, \epsilon) \in \overline{\Sp}(W).$$
Similarly, $\nu_2$ determines an automorphism of $\overline{\Sp}(W)$ iff the following equality holds:
\begin{equation}\label{crao}
c(g, g')=c(g^{\alpha},g^{'\alpha})\nu_2(\alpha,g) \nu_2(\alpha,g') \nu_2(\alpha, gg')^{-1}.
\end{equation}
\begin{lemma}\label{nu2}
$\nu_2( \alpha,g )=\nu(\alpha,g) \frac{m_{X^{\ast}}(g)}{m_{X^{\ast}}(g^{\alpha})}$.
\end{lemma}
\begin{proof}
By (\ref{28inter}),(\ref{crao}),
\begin{equation}
\begin{split}
&c(g, g')m_{X^{\ast}}(gg') m_{X^{\ast}}(g)^{-1}m_{X^{\ast}}(g')^{-1}\\
& =  c_{PR, {X^{\ast}}}(g,g')\\
& =c_{PR, X^{\ast}}(g^{\alpha},g^{'\alpha})\nu(\alpha,g) \nu(\alpha,g') \nu(\alpha,gg')^{-1}\\
&=c(g^{\alpha},g^{'\alpha}) m_{X^{\ast}}(g^{\alpha}g^{'\alpha}) m_{X^{\ast}}(g^{\alpha})^{-1}m_{X^{\ast}}(g^{'\alpha})^{-1}\nu(\alpha,g) \nu(\alpha,g') \nu(\alpha,gg')^{-1}.
\end{split}
\end{equation}
Hence:
$$\frac{c(g, g')}{c(g^{\alpha},g^{'\alpha})}= [\frac{m_{X^{\ast}}(gg')}{m_{X^{\ast}}(g^{\alpha}g^{'\alpha})}]^{-1} \frac{m_{X^{\ast}}(g)}{ m_{X^{\ast}}(g^{\alpha})}
\frac{m_{X^{\ast}}(g')}{ m_{X^{\ast}}(g^{'\alpha})}\nu(\alpha,g) \nu(\alpha,g') \nu(\alpha,gg')^{-1}.$$
So $\nu_2( \alpha,g)=\nu(\alpha,g) \frac{m_{X^{\ast}}(g)}{m_{X^{\ast}}(g^{\alpha})}$.
\end{proof}

Similarly,  we   can obtain a group $F^{\times} \ltimes \overline{\Sp}(W)$ and an exact sequence:
$$ 1\longrightarrow \mu_2 \longrightarrow F^{\times} \ltimes \overline{\Sp}(W) \stackrel{ }{\longrightarrow}F^{\times} \ltimes \Sp(W) \longrightarrow 1.$$
 Let us write $ \overline{\GSp}^B(W)$ for the group $ F^{\times} \ltimes \overline{\Sp}(W)$, and  denote   the corresponding  $2$-cocycle by $C_B(-,-)$. Then:
$$ C_B([y_1,g_1], [y_2, g_2]) =\nu_2(y_2,g_1)c(g_1^{y_2}, g_2),   \quad\quad [y_i, g_i]\in F^{\times} \ltimes \Sp(W).$$
Consequently,
$$C_B([y_1,g_1], [y_2, g_2]) =\nu_2( y_2, g_1)c(g_1^{y_2}, g_2)=\nu(y_2 ,g_1) \frac{m_{X^{\ast}}(g_1)}{m_{X^{\ast}}(g_1^{y_2})}m_{X^{\ast}}(g_1^{y_2}g_2)^{-1} m_{X^{\ast}}(g_1^{y_2}) m_{X^{\ast}}(g_2) c_{PR, {X^{\ast}}}(g_1^{y_2},g_2).$$
Hence:
\begin{equation}\label{29inter}
C_B([y_1,g_1], [y_2, g_2])=m_{X^{\ast}}(g_1^{y_2}g_2)^{-1}m_{X^{\ast}}(g_1)m_{X^{\ast}}(g_2)C_{BPR, X^{\ast}}([y_1,g_1], [y_2, g_2]).
\end{equation}
\subsubsection{Example} Let us consider the case that  $\dim W=2$, $\GSp(W)\simeq \GL_2(F)$.    Let  $h_i=\begin{pmatrix}
a_i& b_i\\
c_i& d_i
\end{pmatrix}\in \GL_2(F)$,  $i=1, 2, 3$, with $h_3=h_1h_2$.    Let us  write $h_i=\begin{pmatrix}
1& 0\\
0& y_i
\end{pmatrix} g_i$, for $y_i=\det h_i$, $g_i=\begin{pmatrix}
a_i& b_i\\
y_i^{-1}c_i& y_i^{-1}d_i
\end{pmatrix}$.   Then:
  $$C_B(h_1, h_2) =\nu_2(y_2,g_1)c(g_1^{y_2}, g_2)=\nu_2(\det h_2,g_1)c(g_1^{\det h_2}, g_2).$$
\begin{itemize}
\item[(1)] If $c_1=0$, $m_{X^{\ast}}(g_1)=m_{X^{\ast}}(g_1^{y_2})$, so $\nu_2(y_2, g_1)=\nu(y_2, g_1)=(y_2,  a_1)_F$.
\item[(2)]  If $m_{X^{\ast}}(g_1)=\gamma_{\psi}(\tfrac{1}{2} c_1 y_1^{-1})^{-1} $, $m_{X^{\ast}}(g_1^{y_2})=\gamma_{\psi}(\tfrac{1}{2} c_1 y_1^{-1}y_2^{-1})^{-1} $.
 $$\frac{m_{X^{\ast}}(g_1)}{  m_{X^{\ast}}(g_1^{y_2})}= \frac{\gamma_{\psi}(\tfrac{1}{2} c_1 y_1^{-1}y_2^{-1})}{\gamma_{\psi}(\tfrac{1}{2} c_1 y_1^{-1})}=\gamma(y^{-1}_2, \psi^{\tfrac{1}{2}}) (y_1c_1,y_2)_F.$$
  $$\nu_2(y_2, g_1)=\nu(y_2, g_1)\frac{m_{X^{\ast}}(g_1)}{  m_{X^{\ast}}(g_1^{y_2})}=(y_1 c_1, y_2)_F\gamma(y_2, \psi^{\frac{1}{2}})^{-1}\gamma(y^{-1}_2, \psi^{\tfrac{1}{2}}) (y_1c_1,y_2)_F=1.$$
  \end{itemize}
  \begin{remark}
  The above $\nu_2(-,-)$ is compatible with Kubota's cocycle on $\GL_2(F)$. For more information on this aspect, readers can refer to the thesis written by Patel in 2014(cf. \cite{Pa}).
  \end{remark}
 \subsection{$F^{\times}$}\label{Ftimes}  Recall $F^{\times} \simeq  \mathfrak{f} \times U_1\times \langle \varpiup\rangle  $.   By \cite{Ge},   $U_1 \subseteq F^{\times 2}$. Since $U_1$ is a pro-p-group($p\neq 2$), the square  map is an automorphism. So there exists a group homomorphism $\sqrt{\quad}: U_1 \longrightarrow U_1$ such that $(\sqrt{a^2})^2=a^2$. On $\langle \varpiup^2\rangle $, we fix  the square root map: $\sqrt{\varpiup^2}=-\varpiup$. Let us consider $\mathfrak{f}$ and its subgroup $\mathfrak{f}^2=\langle \zeta^2\rangle$.  It is known that there exists an exact sequence:
\[
\begin{CD}\label{Ftwo}
1 @>>> \{\pm 1\} @>>> \mathfrak{f} @>>> \mathfrak{f}^{2} @>>> 1,\\
  @.  @. a @>>> a^2. @.
\end{CD}
\]
Hence $\mathfrak{f}$ can be viewed as a central extension of $\mathfrak{f}^{2}$ by $\{\pm 1\}$.  More precisely, let us choose a square root map $\sqrt{\,}: \mathfrak{f}^{ 2} \longrightarrow \mathfrak{f}$, such that $(\sqrt{a^2})^2=a^2$, and $\sqrt{1}=1$. Associated to this map, there is a $2$-cocycle $c_{\sqrt{\,}}$ from $\mathfrak{f}^{2}\times \mathfrak{f}^{2}$ to $\mathfrak{f}$.  Let $\widetilde{\mathfrak{f}^{ 2}}$ be the corresponding central  covering  group over $\mathfrak{f}^{ 2}$.   Then there exists a group isomorphism:
$$\widetilde{\mathfrak{f}^{ 2}} \longrightarrow \mathfrak{f}; [a^2, \epsilon]\longmapsto \sqrt{a^2} \epsilon.$$
Consequently,  $F^{\times}\simeq  \widetilde{\mathfrak{f}^2} \times U_1\times \langle \varpiup\rangle $.
\begin{itemize}\label{q143}
\item[(1)] If   $|k_F|\equiv 3(\bmod4)$, let us write $|k_F|-1=2l$, for some odd number $l$. Then $\mathfrak{f}=\langle \zeta^2\rangle \times\{\pm 1\}=\langle \zeta^4\rangle \times\{\pm 1\}$. Hence we can choose the square root map: $\sqrt{\,}: \mathfrak{f}^2 \longrightarrow \mathfrak{f}; \zeta^4 \longrightarrow \zeta^2$ such that $c_{\sqrt{\,}}\equiv1$.
\item[(2)]  If   $|k_F|\equiv 1(\bmod4)$, let us write $|k_F|-1=2^kl$, for some odd number $l$, and $k>1$.  Assume $\zeta=\zeta_1  \zeta_2$, for an   element  $\zeta_1$ of order $2^k$ and  another element $\zeta_2$ of order $l$. Then $\mathfrak{f}^2=\langle \zeta^2\rangle=\langle \zeta_1^2\rangle \times \langle \zeta_2^2\rangle$. 
    We  fix a  square root map: $$\sqrt{\,}: \mathfrak{f}^2 \longrightarrow \mathfrak{f}; \zeta_1^{2i} \times \zeta_2^{2j} \longrightarrow (-\zeta_1)^{i}\times \zeta_2^{j},$$ for $0\leq i\leq 2^{k-1}-1,  1\leq j\leq l$. Then
    $$c_{\sqrt{\,}}(\zeta_1^{2i_1} \zeta_2^{2j_1}, \zeta_1^{2i_2} \zeta_2^{2j_2})=c_{\sqrt{\,}}(\zeta_1^{2i_1}, \zeta_1^{2i_2} )=\left\{\begin{array}{lc} 1 &  \textrm{ if } i_1+i_2< 2^{k-1},\\
    -1 &  \textrm{ if }  i_1+i_2\geq 2^{k-1},
    \end{array}\right.$$
    for $0\leq i_1,i_2 \leq 2^{k-1}-1$, $1\leq j_1,j_2 \leq l$.
\end{itemize}
\section{The group $\PGMp^{\pm}(W)$: Case $|k_F|\equiv 3(\bmod4)$}\label{case3mod42}
Following Barthel's work,   in this section we will construct  the group $\PGSp^{\pm}(W)$  and its Metaplectic $8$-degree and $2$-degree covering groups.
\subsection{$\widetilde{F^{\times}}$}\label{sectionmap}  For our purpose,  we  select  a section map $\iota$ from $F^{\times}$ to $\GSp(W)$ in the following way:
\begin{align}
\iota:  U_1\longrightarrow  \GSp(W); &u \longmapsto
\begin{pmatrix}
\sqrt{u} & 0\\
0 & \sqrt{u} \end{pmatrix}, \\
\iota:  \langle \varpiup\rangle    \longrightarrow  \GSp(W);& \varpiup \longmapsto \begin{pmatrix}
0& -I\\
\varpiup I& 0\end{pmatrix},\\
\iota: \mathfrak{f}^2   \longrightarrow  \GSp(W);&
   x\longmapsto  \begin{pmatrix}
\sqrt{x} & 0\\
0 & \sqrt{x} \end{pmatrix},\\
\iota:\{\pm 1\}   \longrightarrow  \GSp(W);& -1\longmapsto \begin{pmatrix}
I& 0\\
0& -I\end{pmatrix}.
\end{align}
For $a= a_1a_2a_3$ with $a_1\in \mathfrak{f} $, $a_2\in U_1$,   $a_3\in \langle \varpiup\rangle$, let us define
 $$\iota(a)=\iota(a_1)\iota(a_2)\iota(a_3).$$
The restriction of $\iota$ on $U_1\times \langle \varpiup\rangle $ and $\mathfrak{f}$ are both group homomorphisms, but $\iota$ is not a group homomorphism on the entire group $F^{\times}$. Let $\widetilde{F^{\times}}$ denote the group generated by these $\iota(F^{\times})$ in $\GSp(W)$.
 Let us define
\begin{align}
  D=\widetilde{F^{\times}} \cap \Sp(W).
\end{align}
\begin{lemma}\label{mod1}
\begin{itemize}
\item[(1)] $F^{\times} \subseteq \widetilde{F^{\times}} $.
 \item[(2)] $D=\langle \begin{pmatrix}
-I &   0\\
 0& -I\end{pmatrix}\rangle\simeq \mu_2$.
\end{itemize}
 \end{lemma}
\begin{proof}
If $y\in F^{\times 2}$, then $\iota(y)$ belongs to the center  group of $\GSp(W)$.  $$\iota(-\varpiup)=\iota(-1)\iota(\varpiup)=\begin{pmatrix}
I& 0\\
 0& -I\end{pmatrix}\begin{pmatrix}
0& -I\\
\varpiup I& 0\end{pmatrix}=\begin{pmatrix}
0& -I\\
-\varpiup I& 0\end{pmatrix};$$
$$\iota(\varpiup)\iota(-1)=\begin{pmatrix}
0& -I\\
\varpiup I& 0\end{pmatrix}\begin{pmatrix}
I& 0\\
 0& -I\end{pmatrix}=\begin{pmatrix}
0& I\\
\varpiup I& 0\end{pmatrix}=-\iota(-1)\iota(\varpiup).$$
\end{proof}
Then there exists an  exact sequence:
\begin{equation}\label{exact2}
1 \longrightarrow D  \longrightarrow \widetilde{F^{\times}} \stackrel{\lambda}{\longrightarrow} F^{\times} \longrightarrow 1.
\end{equation}
We remark that $D\subseteq M$ (cf. Section \ref{In}).  Note that
$$\widetilde{F^{\times}}=F^{\times} \cup F^{\times}  \begin{pmatrix}
I& 0\\
0& -I\end{pmatrix} \cup F^{\times}  \begin{pmatrix}
0& -I\\
\varpiup I& 0\end{pmatrix} \cup  F^{\times} \begin{pmatrix}
0& -I\\
-\varpiup I& 0\end{pmatrix}.$$
There exists an exact sequence:
\begin{equation}\label{exact211}
1 \longrightarrow F^{\times}  \longrightarrow \widetilde{F^{\times}} \stackrel{\dot{\lambda}}{\longrightarrow} F^{\times}/F^{\times 2} \longrightarrow 1.
\end{equation}
For any $y\in \widetilde{F^{\times}}$, $g\in \Sp(W)$, let  us define $g^{y}=y^{-1}gy$. Under this conjugacy action, we obtain a semi-product group $\widetilde{F^{\times}} \ltimes \Sp(W)$, and an  exact sequence:
$$1 \longrightarrow \Delta_{D} \longrightarrow \widetilde{F^{\times}} \ltimes \Sp(W) \longrightarrow \GSp(W) \longrightarrow 1,$$
where $\Delta_{D}=\{ (y, y^{-1})\mid y\in D\}$.  
\begin{definition}
$ \PGSp^{\pm}(W) =[\widetilde{F^{\times}} \ltimes \Sp(W)]/[F^{\times} \times 1].$
\end{definition}
Then there exists an exact sequence:
$$1 \longrightarrow \Sp(W)  \to \PGSp^{\pm}(W) \to F^{\times}/F^{\times 2} \longrightarrow 1.$$
\subsection{} Barthel's work allows us to transfer the  conjugacy  action onto  $\Mp(W)$ and   obtain a group $\widetilde{F^{\times}} \ltimes \Mp(W)$, as well as an exact sequence:
$$ 1\longrightarrow \mu_8 \longrightarrow \widetilde{F^{\times}} \ltimes \Mp(W) \stackrel{\widetilde{p}}{\longrightarrow} \widetilde{F^{\times}} \ltimes \Sp(W) \longrightarrow 1.$$
Let $\widetilde{C}_M(-,-)$ denote the $2$-cocycle   associated to this exact sequence. For  $([y_1, g_1], \epsilon_1), ([y_2, g_2], \epsilon_2)\in \widetilde{F^{\times}} \ltimes \Mp(W)$,
\begin{equation}
\begin{split}
(y_1, g_1, \epsilon_1)\cdot (y_2, g_2, \epsilon_2)&=(y_1y_2, [g_1, \epsilon_1]^{y_2}[g_2, \epsilon_2])\\
& =(y_1y_2, [g_1^{y_2}, \nu(y_2, g_1)\epsilon_1][g_2, \epsilon_2])\\
&=(y_1y_2, [g_1^{y_2}g_2, \nu(y_2, g_1)c_{PR, X^{\ast}}(g_1^{y_2}, g_2)\epsilon_1\epsilon_2]).
\end{split}
\end{equation}
Hence:
\begin{equation}\label{equations4mod1}
\widetilde{C}_M([y_1,g_1], [y_2, g_2]) =\nu( y_2, g_1)c_{PR, X^{\ast}}(g_1^{y_2}, g_2),   \quad\quad [y_i, g_i]\in \widetilde{F^{\times}} \ltimes \Sp(W).
\end{equation}
\subsection{$\nu(-,-)$} Let's find out  the explicit expression of $\nu(-,-)$.
\begin{itemize}
 \item If $y\in D$,  by Example \ref{example1},  $\nu(y,g)=c_{PR, X^{\ast}}(y^{-1}, g y)c_{PR, X^{\ast}}(g,y)=1$.
 \item If $x=zy=yz$, for some $z\in \widetilde{F^{\times}}$, and $y\in D$,  then  by Lemma \ref{twoau},   $\nu(x,g)=\nu(z,g)\nu(y,g^z)=\nu(z,g)$.
 \item If $y\in \iota(F^{\times 2})$,  then $g^y=g$,  so $\nu(y,-)$ is a character of $\Sp(W)$, and then $\nu(y,g)\equiv 1$.
 \item If $x=zy=yz$,   for some $z\in \widetilde{F^{\times}}$, and $y\in  \iota(F^{\times 2})$,  then  $\nu(x,g)=\nu(z,g)\nu(y,g^z)=\nu(z,g)$.
 \item $\nu(-, -)$ is a function on $(F^{\times}/F^{\times 2})\times \Sp(W)$.
 \end{itemize}
Note that $F^{\times}/F^{\times 2}$ can be represented by $\{ \dot{1}, \dot{(-1)}, \dot{\varpiup}, \dot{(-\varpiup)}\}$.
\begin{lemma}\label{nu1}
 $$\nu(y,g)=\left\{ \begin{array}{ll}
 (\det a_1a_2, -1)_F \gamma(-1,\psi^{\frac{1}{2}})^{-|S|} &\textrm{ if } y=\iota(-1),\\
(\det a_1a_2, \varpiup)_F \gamma(\varpiup,\psi^{\frac{1}{2}})^{-|S|}c_{PR, X^{\ast}}(\omega^{-1}, g^{y_1} \omega)c_{PR, X^{\ast}}(g,\omega)& \textrm{ if } y=\iota(\varpiup)=y_1\omega,\\
(\det a_1a_2, -\varpiup)_F \gamma(-\varpiup,\psi^{\frac{1}{2}})^{-|S|}c_{PR, X^{\ast}}(\omega^{-1}, g^{y_2} \omega)c_{PR, X^{\ast}}(g,\omega)&\textrm{ if } y=\iota(-\varpiup)=y_2\omega,
\end{array}\right.$$
for $g=p_1\omega_S p_2$, $p_i=\begin{pmatrix}
a_i& b_i\\
0& d_i
\end{pmatrix}$, $i=1,2$.
\end{lemma}
\begin{proof}
1) It follows from (\ref{alpp}).\\
2)  By Examples \ref{example1},\ref{example2} and Lemma \ref{twoau},
 $$\nu(y,g)=\nu(y_1,g)\nu(\omega, g^{y_1})=(\det a_1a_2, \varpiup)_F \gamma(\varpiup, \psi^{\frac{1}{2}})^{-|S|}c_{PR, X^{\ast}}(\omega^{-1}, g^{y_1} \omega)c_{PR, X^{\ast}}(g,\omega).$$
 3) Similarly to (2).
\end{proof}

\begin{lemma}
$\nu(y,\omega_S)=\gamma(\lambda_{y},\psi^{\frac{1}{2}})^{-|S|}$.
\end{lemma}
\begin{proof}
If $\dot{y}=\dot{1} \in F^{\times}/F^{\times2}$, then the result holds. If $y=\iota(\varpiup)=y_1\omega$, then $$\nu(y,\omega_S)= \gamma(\varpiup,\psi^{\frac{1}{2}})^{-|S|}c_{PR, X^{\ast}}(\omega^{-1}, \omega_S^{y_1} \omega)c_{PR, X^{\ast}}(\omega_S,\omega)=\gamma(\varpiup,\psi^{\frac{1}{2}})^{-|S|}.$$
This applies to the other elements as well.
\end{proof}
\begin{lemma}\label{gg0mod43}
For $g_0=\begin{pmatrix}
a& 0\\
0& (a^{\ast})^{-1}
\end{pmatrix}\in M$, $g\in \Sp(W)$, $y\in \widetilde{F^{\times}}$, $\lambda_y\in F^{\times}$, we have:
\begin{itemize}
\item[(1)] $\nu(y, g_0)=(\det a,\lambda_y)_F$,
\item[(2)] $\nu(y, g_0g)=\nu(y, g_0) \nu(y, g)$,
\item[(3)] $\nu(y, gg_0)= \nu(y, g)\nu(y, g_0)$.
\end{itemize}
\end{lemma}
\begin{proof}
1)  It can be deduced from the formulas presented in Lemma \ref{nu1}.  \\
2) Let us write $g=p_1\omega_S p_2$.   Then $g_0g=g_0p_1\omega_S p_2$. So the result is right for $y=\iota(-1)$.   For $y=\iota(\varpiup)=y_1\omega$,
$$\nu(y,g)=(\det aa_1a_2, \varpiup)_F \gamma(\varpiup,\psi^{\frac{1}{2}})^{-|S|}c_{PR, X^{\ast}}(\omega^{-1}, (g_0g)^{y_1} \omega)c_{PR, X^{\ast}}(g_0g,\omega).$$
Note:
\begin{equation}\label{PR1}
\begin{split}
c_{PR, X^{\ast}}(\omega^{-1}, (g_0g)^{y_1} \omega)& =c_{PR, X^{\ast}}(\omega^{-1}, g_0^{y_1})c_{PR, X^{\ast}}(\omega^{-1}, (g_0g)^{y_1} \omega)\\
&=c_{PR, X^{\ast}}(\omega^{-1} g_0^{y_1},g^{y_1} \omega )c_{PR, X^{\ast}}( g_0^{y_1}, g^{y_1} \omega)\\
&=c_{PR, X^{\ast}}(\omega^{-1} g_0^{y_1},g^{y_1} \omega )=c_{PR, X^{\ast}}([\omega^{-1} g_0^{y_1}\omega] \omega^{-1},g^{y_1} \omega )\\
&=c_{PR, X^{\ast}}([\omega^{-1} g_0^{y_1}\omega] \omega^{-1},g^{y_1} \omega )c_{PR, X^{\ast}}([\omega^{-1} g_0^{y_1}\omega], \omega^{-1} )\\
&=c_{PR, X^{\ast}}([\omega^{-1} g_0^{y_1}\omega], \omega^{-1}g^{y_1} \omega  )c_{PR, X^{\ast}}(\omega^{-1}, g^{y_1} \omega  )\\
&=c_{PR, X^{\ast}}(\omega^{-1}, g^{y_1} \omega  );
\end{split}
\end{equation}
\begin{equation}\label{PR2}
\begin{split}
c_{PR, X^{\ast}}(g_0g,\omega)=c_{PR, X^{\ast}}(g,\omega).
\end{split}
\end{equation}
Hence, the result is right for $y=\iota(\varpiup)$.  If $y=\iota(-\varpiup)=\iota(-1)\iota(\varpiup)$,
by Lemma \ref{twoau}, $$\nu(y, g_0g)=\nu(-1, g_0g)\nu(\varpiup, (g_0g)^{\iota(-1)})$$
$$=\nu(-1, g_0)\nu(\varpiup, g_0^{\iota(-1)})\nu(-1, g)\nu(\varpiup, g^{\iota(-1)})$$
$$=\nu(y, g_0)\nu(y, g).$$
3)The proof is similar to the one above. Let us verify the same equalities as (\ref{PR1}) and (\ref{PR2}).
\begin{equation}\label{PR3}
\begin{split}
c_{PR, X^{\ast}}(\omega^{-1}, (gg_0)^{y_1} \omega)&=c_{PR, X^{\ast}}(\omega^{-1}, [g^{y_1} \omega]\cdot [\omega^{-1}g_0^{y_1}\omega])\\
&=c_{PR, X^{\ast}}(\omega^{-1}, [g^{y_1} \omega]\cdot [\omega^{-1}g_0^{y_1}\omega])c_{PR, X^{\ast}}(g^{y_1} \omega,\omega^{-1}g_0^{y_1}\omega)\\
&=c_{PR, X^{\ast}}(\omega^{-1}, g^{y_1} \omega)c_{PR, X^{\ast}}(\omega^{-1}g^{y_1} \omega,\omega^{-1}g_0^{y_1}\omega)\\
&=c_{PR, X^{\ast}}(\omega^{-1}, g^{y_1} \omega);
\end{split}
\end{equation}
\begin{equation}\label{PR4}
\begin{split}
c_{PR, X^{\ast}}(gg_0,\omega)&=c_{PR, X^{\ast}}(g,g_0)c_{PR, X^{\ast}}(gg_0,\omega)\\
&=c_{PR, X^{\ast}}(g, g_0 \omega)c_{PR, X^{\ast}}(g_0,  \omega)\\
&=c_{PR, X^{\ast}}(g, \omega [\omega^{-1}g_0 \omega])\\
&=c_{PR, X^{\ast}}(g, \omega [\omega^{-1}g_0 \omega])c_{PR, X^{\ast}}(\omega, [\omega^{-1}g_0 \omega])\\
&=c_{PR, X^{\ast}}(g, \omega) c_{PR, X^{\ast}}(g \omega,\omega^{-1}g_0 \omega)\\
&=c_{PR, X^{\ast}}(g, \omega).
\end{split}
\end{equation}
\end{proof}
\subsubsection{Example} If $\dim W=2$, $\Sp(W)\simeq \SL_2(F)$. Recall $\dot{\lambda}: \widetilde{F^{\times}} \longrightarrow F^{\times}/F^{\times 2}$, and  $\nu(-, -)$ is a function on $F^{\times}/F^{\times 2} \times \Sp(W)$. By Lemma \ref{nu1},
$$\nu(y,g)=\left\{ \begin{array}{ll}
(a, -1)_F &\textrm{ if } \dot{(\lambda_y)}=\dot{(-1)},\\
(a, \varpiup)_F\gamma_{\psi}(\tfrac{1}{2} ab\varpiup)& \textrm{ if } \dot{(\lambda_y)}=\dot{\varpiup},\\
(a,-\varpiup)_F\gamma_{\psi}(-\tfrac{1}{2} ab\varpiup) &\textrm{ if } \dot{(\lambda_y)}=\dot{(-\varpiup)},
\end{array}\right. \quad\quad\quad  \textrm{ for }g=\begin{bmatrix} a& b \\ 0 & a^{-1}\end{bmatrix};$$
$$\nu(y,g)=\left\{ \begin{array}{ll}
(c, -1)_F \gamma(-1,\psi^{\frac{1}{2}})^{-1} &\textrm{ if } \dot{(\lambda_y)}=\dot{(-1)},\\
(c, \varpiup)_F\gamma(\varpiup,\psi^{\frac{1}{2}})^{-1}\gamma_{\psi}(\tfrac{1}{2} bd\varpiup)\gamma_{\psi}(\tfrac{1}{2} cd)& \textrm{ if } \dot{(\lambda_y)}=\dot{\varpiup},\\
(c,-\varpiup)_F\gamma(-\varpiup,\psi^{\frac{1}{2}})^{-1}\gamma_{\psi}(-\tfrac{1}{2} bd\varpiup)\gamma_{\psi}(\tfrac{1}{2} cd) &\textrm{ if } \dot{(\lambda_y)}=\dot{(-\varpiup)},
\end{array}\right. \quad \quad \quad \textrm{ for }g=\begin{bmatrix} a& b \\ c & d\end{bmatrix}, c\neq 0.$$

\subsection{$\PGMp^{\pm}(W)$}
\begin{lemma}\label{cocomod41}
For $[y_1, g_1]\in \Delta_D$, with $g_1=y_1^{-1}= \begin{bmatrix}
t& \\
& t^{-1}\end{bmatrix}$,  $[y_2, g_2]\in \widetilde{F^{\times}} \ltimes \Sp(W)$,
$\widetilde{C}_M([y_1,g_1], [y_2, g_2]) =( t^m, \lambda_{y_2})_F$, and $\widetilde{C}_M([y_2,g_2], [y_1, g_1]) = 1$.
\end{lemma}
\begin{proof}
By (\ref{equations4mod1}), $$\widetilde{C}_M([y_1,g_1], [y_2, g_2])=\nu( y_2, g_1)c_{PR, X^{\ast}}(g_1^{y_2}, g_2)=\nu(y_2, g_1)=( t^m, \lambda_{y_2})_F,$$
 $$\widetilde{C}_M([y_2,g_2], [y_1, g_1])=\nu( y_1, g_2   )c_{PR, X^{\ast}}(g_2^{y_1}, g_1)=\nu( y_1, g_2   )=1.$$
\end{proof}
\begin{corollary}
$\widetilde{C}_M([y_2,g_2], [y_1, g_1]) = 1$, for $[y_i, g_i]\in \Delta_D$.
\end{corollary}
Hence, there exists the following exact sequence:
\begin{align}\label{mu8}
1 \longrightarrow \mu_8\times \Delta_D \longrightarrow \widetilde{F^{\times}} \ltimes \Mp(W) \stackrel{p_1}{\longrightarrow} \GSp(W) \longrightarrow 1.
\end{align}

\begin{lemma}\label{cent3mdulo4}
Let $Z$ denote the center group  of $\widetilde{F^{\times}} \ltimes \Mp(W)$. Then:
 $Z=\left\{\begin{array}{lr} (F^{\times} \times 1) \times \mu_8 & 2\nmid m,\\  (F^{\times} \times \{\pm 1\})\times  \mu_8 &  2\mid m. \end{array}\right. $
\end{lemma}
\begin{proof}
 For any $[y_1, g_1]\in F^{\times} \times 1$, $[y_2, g_2]\in \widetilde{F^{\times}} \ltimes \Sp(W)$,
 $$\widetilde{C}_M([y_1,g_1], [y_2, g_2])=\nu( y_2, g_1)c_{PR, X^{\ast}}(g_1^{y_2}, g_2)=\nu( y_2, g_1)c_{PR, X^{\ast}}(g_1, g_2)=1;$$
$$\widetilde{C}_M([y_2,g_2], [y_1, g_1])=\nu( y_1, g_2)c_{PR, X^{\ast}}(g_2^{y_1}, g_1)=\nu( y_1, g_2)=1.$$
Hence $ (F^{\times}\times 1) \times \mu_8 $ belongs to the centers of  $\widetilde{F^{\times}} \ltimes \Mp(W)$.  Conversely, the other possible  central elements should be $[y_1, -1]$, for $y_1\in F^{\times}$. Then:
$$\widetilde{C}_M([y_2,g_2], [y_1, -1])=\nu( y_1, g_2)=1, \quad \widetilde{C}_M([y_1,-1], [y_2, g_2])=(\lambda_{y_2}, (-1)^m)_F.$$
So the result holds.
\end{proof}
From the proof, we have:
$$\widetilde{C}_M([y_1,g_1], [y_2, g_2])=1=\widetilde{C}_M([y_2,g_2], [y_1, g_1])$$
for $[y_1,g_1]\in F^{\times} \times 1$. Therefore, $\widetilde{C}_M(-,-)$ also defines a cocycle on $\PGSp^{\pm}(W)$. Let us define:
$$\PGMp^{\pm}(W)=[\widetilde{F^{\times}} \ltimes \Mp(W)]/[F^{\times} \times 1_{\Sp(W)}].$$
Then $\PGMp^{\pm}(W)$ is a central extension of $\PGSp^{\pm}(W)$ by $\mu_8$ associated with $\widetilde{C}_M(-,-)$. Moreover, there exists an exact sequence:
$$1 \to \Mp(W) \to \PGMp^{\pm}(W) \stackrel{\dot{\lambda}}{\longrightarrow} F^{\times}/F^{\times 2} \to 1.$$
\subsubsection{Example} Consider $\dim W=2$, $\Sp(W)\simeq \SL_2(F)$. Write 
$$\iota(1)= \begin{pmatrix}
1&0\\
0& 1\end{pmatrix}, \iota(-1)= \begin{pmatrix}
1&0\\
0& -1\end{pmatrix},  \iota(\varpiup)= \begin{pmatrix}
0& -1\\
\varpiup & 0\end{pmatrix}, \iota(-\varpiup)=\begin{pmatrix}
0& -1\\
-\varpiup & 0\end{pmatrix}.$$   Then $$\PGSp^{\pm}(W)=\{ [\iota(t), g]\mid t=-1, 1, \varpiup, -\varpiup, \textrm{ and } g \in \SL_2(F)\}.$$ 
$$\widetilde{C}_M([y_1,g_1], [y_2, g_2])=\nu(y_2, g_1)c_{PR, X^{\ast}}(g_1^{y_2}, g_2).$$ 
\subsection{ $\overline{\PGSp^{\pm}}(W)$}\label{case3mod43}
We can also  lift the action of $\widetilde{F^{\times}}$ from  $\Sp(W)$ to $\overline{\Sp}(W)$, and then  obtain a group $\widetilde{F^{\times}} \ltimes \overline{\Sp}(W)$ as well as an exact sequence:
$$ 1\longrightarrow \mu_2 \longrightarrow \widetilde{F^{\times}} \ltimes \overline{\Sp}(W) \stackrel{\widetilde{p}}{\longrightarrow} \widetilde{F^{\times}} \ltimes \Sp(W) \longrightarrow 1.$$
Let $\overline{C}_M(-,-)$ denote the $2$-cocycle   associated to this exact sequence. Similarly,
\begin{equation}\label{equations2}
\overline{C}_M([y_1,g_1], [y_2, g_2]) =\nu_2( y_2, g_1)c(g_1^{y_2}, g_2),   \quad\quad [y_i, g_i]\in \widetilde{F^{\times}} \ltimes \Sp(W).
\end{equation}

Let us also give the explicit expression of $\nu_2(y_2, g_1)$.  For $y\in \widetilde{F^{\times}}$, $g\in \Sp(W)$, according to Lemma \ref{nu2},
$$\nu_2( y, g)=\nu(y,g) \frac{m_{X^{\ast}}(g)}{m_{X^{\ast}}(g^{y})}.$$

\begin{lemma}
 $\nu_2(-,-)$ is a function on $F^{\times }/F^{\times 2}\times \Sp(W)$.
\end{lemma}
\begin{proof}
For $y\in  F^{\times}$, by (\ref{mx}), $m_{X^{\ast}}(g^y)=m_{X^{\ast}}(g)$, then  $\nu_2( y, g)=\nu(y,g)$, for such $y$.
 If $x=zy$, for some $z\in \widetilde{F^{\times}}$, and $y\in  F^{\times}$,  then  by Lemma \ref{twoau},   $\nu_2(x,g)=\nu_2(z,g)\nu_2(y,g^z)=\nu_2(z,g)$.
If $x=yz=z(z^{-1}yz)$, for some $z\in \widetilde{F^{\times}}$, and $y\in  F^{\times}$,  then  by Lemma \ref{twoau},   $\nu_2(x,g)=\nu_2(z,g)\nu_2( z^{-1}yz,g^{z})=\nu_2(z,g)$.
\end{proof}
\begin{lemma}
For $g_0=\begin{pmatrix}
a& 0\\
0& (a^{\ast})^{-1}
\end{pmatrix}\in M$, $g\in \Sp(W)$, $y\in \widetilde{F^{\times}}$, $\lambda_y\in F^{\times}$, we have:
\begin{itemize}
\item[(1)] $\nu_2(y, g_0)=(\det a,\lambda_y)_F$,
\item[(2)] $\nu_2(y, g_0g)=\nu_2(y, g_0) \nu_2(y, g)(\det a, x(g)x(g^y))_F$,
\item[(3)] $\nu_2(y, gg_0)= \nu_2(y, g)\nu_2(y, g_0)(\det a, x(g)x(g^y))_F$.
\end{itemize}
\end{lemma}
\begin{proof}
1) By (\ref{mx}), $m_{X^{\ast}}(g_0^y)=m_{X^{\ast}}(g_0)$, so $ \nu_2(y, g_0)=\nu(y, g_0)=(\det a,\lambda_y)_F$. \\
2) By (\ref{mx}), $m_{X^{\ast}}(g_0g)=m_{X^{\ast}}(g_0)m_{X^{\ast}}(g)(\det a, x(g))_F$. Since $g_0^y \in M$,
\begin{align*}
m_{X^{\ast}}(g_0^yg^y)&=m_{X^{\ast}}(g_0^y)m_{X^{\ast}}(g^y)(\det a, x(g^y))_F\\
&=m_{X^{\ast}}(g_0)m_{X^{\ast}}(g^y)(\det a, x(g^y))_F.
\end{align*}
Hence:
\begin{align*}
\nu_2(y, g_0g)&=\nu(y,g_0g) \frac{m_{X^{\ast}}(g_0g)}{m_{X^{\ast}}(g_0^yg^{y})}\\
&=\nu(y, g_0) \nu(y, g) \frac{m_{X^{\ast}}(g_0)m_{X^{\ast}}(g)(\det a, x(g))_F}{m_{X^{\ast}}(g_0^y)m_{X^{\ast}}(g^y)(\det a, x(g^y))_F}\\
&=\nu_2(y, g_0) \nu_2(y, g)(\det a, x(g)x(g^y))_F.
\end{align*}
3) The proof is similar as (2).
\end{proof}
\begin{lemma}
\begin{align}\label{comp}
\overline{C}_M([y_1,g_1], [y_2, g_2])=\widetilde{C}_M([y_1,g_1], [y_2, g_2])m_{X^{\ast}}(g_1^{y_2}g_2)^{-1}m_{X^{\ast}}(g_1) m_{X^{\ast}}(g_2).
\end{align}
\end{lemma}
\begin{proof}
By Lemma \ref{nu2}, $\nu_2(y_2, g_1)=\nu(y_2,g_1) \frac{m_{X^{\ast}}(g_1)}{m_{X^{\ast}}(g_1^{y_2})}$. By (\ref{28inter}), $$c(g_1^{y_2}, g_2)=m_{X^{\ast}}(g_1^{y_2}g_2)^{-1} m_{X^{\ast}}(g_1^{y_2}) m_{X^{\ast}}(g_2) c_{PR, {X^{\ast}}}(g_1^{y_2},g_2).$$
Hence:
\begin{align*}
\overline{C}_M([y_1,g_1], [y_2, g_2])&=\nu(y_2,g_1) \frac{m_{X^{\ast}}(g_1)}{m_{X^{\ast}}(g_1^{y_2})}m_{X^{\ast}}(g_1^{y_2}g_2)^{-1} m_{X^{\ast}}(g_1^{y_2}) m_{X^{\ast}}(g_2) c_{PR, {X^{\ast}}}(g_1^{y_2},g_2)\\
&=\widetilde{C}_M([y_1,g_1], [y_2, g_2])m_{X^{\ast}}(g_1^{y_2}g_2)^{-1}m_{X^{\ast}}(g_1) m_{X^{\ast}}(g_2).
\end{align*}
\end{proof}
\begin{lemma}\label{coco1}
For $[y_1, g_1]\in \Delta_D$,  $g_1=y_1^{-1}= \begin{bmatrix}
t& \\
& t^{-1}\end{bmatrix}$,  $[y_2, g_2]\in \widetilde{F^{\times}} \ltimes \Sp(W)$, $g_2=p_1\omega_S p_2$, $p_i=\begin{pmatrix}
a_i& b_i\\
0& d_i
\end{pmatrix}$, $i=1,2$,
$$\overline{C}_M([y_1,g_1], [y_2, g_2])=( t^m, \lambda_{y_2})_F(t^m, \det a_1a_2)_F, \quad \quad \overline{C}_M([y_2,g_2], [y_1, g_1]) =(t^m, \det a_1a_2)_F.$$
\end{lemma}
\begin{proof}
By (\ref{comp}),
\begin{align*}
\overline{C}_M([y_1,g_1], [y_2, g_2])& =\widetilde{C}_M([y_1,g_1], [y_2, g_2])m_{X^{\ast}}(g_1^{y_2}g_2)^{-1}m_{X^{\ast}}(g_1) m_{X^{\ast}}(g_2)\\
&=\widetilde{C}_M([y_1,g_1], [y_2, g_2])m_{X^{\ast}}(g_1^{y_2})^{-1}m_{X^{\ast}}(g_2)^{-1}(t^m, x(g_2))_Fm_{X^{\ast}}(g_1) m_{X^{\ast}}(g_2)\\
&= ( t^m, \lambda_{y_2})_F(t^m, x(g_2))_F\\
&=( t^m, \lambda_{y_2})_F(t^m, \det a_1a_2)_F.
\end{align*}
\begin{align*}
\overline{C}_M([y_2,g_2], [y_1, g_1])& =\widetilde{C}_M([y_2,g_2], [y_1, g_1])m_{X^{\ast}}(g_2^{y_1}g_1)^{-1}m_{X^{\ast}}(g_1) m_{X^{\ast}}(g_2)\\
&=\widetilde{C}_M([y_2,g_2], [y_1, g_1])m_{X^{\ast}}(g_1g_2)^{-1}m_{X^{\ast}}(g_1) m_{X^{\ast}}(g_2)\\
&=(t^m, x(g_2))_F=(t^m, \det a_1a_2)_F.
\end{align*}
\end{proof}
\begin{lemma}\label{equalD}
\begin{itemize}
\item[(1)] $\overline{C}_M([y_1,g_1], [y_2, g_2])=1=\overline{C}_M([y_2,g_2], [y_1, g_1])$, for any $[y_1,g_1]\in F^{\times} \times 1$.
\item[(2)] The central group of $\widetilde{F^{\times}}\ltimes\overline{\Sp}(W)$ contains $[F^{\times} \times 1] \times \mu_2$.
\end{itemize}
\end{lemma}
\begin{proof}
1) By (\ref{comp}), 
\begin{align*}
\overline{C}_M([y_1,g_1], [y_2, g_2])&=\widetilde{C}_M([y_1,g_1], [y_2, g_2])m_{X^{\ast}}(g_1^{y_2}g_2)^{-1}m_{X^{\ast}}(g_1) m_{X^{\ast}}(g_2)\\
&= \widetilde{C}_M([y_1,g_1], [y_2, g_2])m_{X^{\ast}}(g_2)^{-1} m_{X^{\ast}}(g_2)   \\
&=\widetilde{C}_M([y_1,g_1], [y_2, g_2])=1;
\end{align*}
\begin{align*}
\overline{C}_M([y_2,g_2], [y_1, g_1])&=\widetilde{C}_M([y_2,g_2], [y_1, g_1])m_{X^{\ast}}(g_2^{y_1}g_1)^{-1}m_{X^{\ast}}(g_2) m_{X^{\ast}}(g_1)\\
&= \widetilde{C}_M([y_2,g_2], [y_1, g_1])m_{X^{\ast}}(g_2^{y_1})^{-1}m_{X^{\ast}}(g_2)   \\
&=\widetilde{C}_M([y_2,g_2], [y_1, g_1])=1.
\end{align*}
2) It is a consequence of (1).
\end{proof}
Let us define:
$$\overline{\PGSp^{\pm}}(W)=[\widetilde{F^{\times}}\ltimes\overline{\Sp}(W)]/[F^{\times} \times 1_{\Sp(W)}].$$
Then $\overline{\PGSp^{\pm}}(W)$ is a central extension of $\PGSp^{\pm}(W)$ by $\mu_2$ associated with $\overline{C}_M(-,-)$. Moreover, there exists an exact sequence:
$$1 \to \overline{\Sp}(W) \to \overline{\PGSp^{\pm}}(W) \stackrel{\dot{\lambda}}{\longrightarrow} F^{\times}/F^{\times 2} \to 1.$$
\section{Extended  Weil representation: Case $|k_F|\equiv 3(\bmod4)$}\label{case3mod44} 
\subsection{A twisted action}\label{atwstamod134}
Recall: 
\begin{equation}\label{fff1}
1\longrightarrow F^{\times } \longrightarrow \widetilde{F^{\times}} \longrightarrow F^{\times}/F^{\times 2} \longrightarrow 1.
\end{equation}
For an element $t\in \widetilde{F^{\times}}$, let $\dot{t}$  denote its reduction in $F^{\times}/F^{\times 2}$. Note that  $F^{\times}/F^{\times 2}$ can be represented by $\{ \dot{1}, \dot{(-1)}, \dot{\varpiup}, \dot{(-\varpiup)}\}$.  Let $\kappa$ be the canonical section map from $F^{\times}/F^{\times 2}$ to $\widetilde{F^{\times}} $ such that 
$$\kappa( \dot{1})= \begin{pmatrix}
I&0\\
0& I\end{pmatrix}, \kappa(\dot{(-1)})= \begin{pmatrix}
I&0\\
0& -I\end{pmatrix},  \kappa( \dot{\varpiup})= \begin{pmatrix}
0& -I\\
\varpiup I & 0\end{pmatrix},\kappa( \dot{(-\varpiup)})=\begin{pmatrix}
0& -I\\
-\varpiup I& 0\end{pmatrix}.$$
$\widetilde{F^{\times}}$ can be viewed as a central extension of $F^{\times}/F^{\times 2}$ by $F^{\times }$. Let $c'(-,-)$ denote the $2$-cocycle associated to $\kappa$. More precisely, for two elements $t_1, t_2\in \widetilde{F^{\times}}$, let us write $t_1=a_1 \kappa(\dot{t}_1)$, $t_2=a_2 \kappa(\dot{t}_2)$. Then:
\begin{align}\label{tt12}
t_1t_2=a_1a_2\kappa(\dot{t}_1) \kappa(\dot{t}_2)=a_1a_2 c'(\dot{t}_1, \dot{t}_2)\kappa(\dot{t}_1\dot{t}_2).
\end{align}
 Then:
 $$\left\{ \begin{array}{ll}
 c'(\dot{t}, \dot{1})=c'(\dot{1}, \dot{t})=c'(\dot{(-1)}, \dot{t})=1,  & \textrm{ for  any }\dot{t}\in F^{\times}/F^{\times 2}; \\
 c'(\dot{\varpiup}, \dot{(-1)})= c'(\dot{(-\varpiup)}, \dot{(-1)})=-1; & \\
  c'(\dot{\varpiup}, \dot{\varpiup})=-\varpiup=c'(\dot{(-\varpiup)}, \dot{\varpiup}); & \\
  c'(\dot{\varpiup}, \dot{(-\varpiup)})= \varpiup = c'(\dot{(-\varpiup)}, \dot{(-\varpiup)}). &
\end{array}\right.$$
Via the projections: $\widetilde{F^{\times}} \longrightarrow F^{\times}/F^{\times 2}$, $\widetilde{F^{\times}}\ltimes \Sp(W)  \longrightarrow F^{\times}/F^{\times 2}$, $\PGSp^{\pm}(W) \longrightarrow F^{\times}/F^{\times 2}$, we can  view $c'(-,-)$ as a cocycle on $\widetilde{F^{\times}}$, $\widetilde{F^{\times}}\ltimes \Sp(W)$, $\PGSp^{\pm}(W)$.  Let $\wideparen{F^{\times}/F^{\times 2}}$, $\wideparen{\widetilde{F^{\times}}}$,   $\wideparen{\widetilde{F^{\times}}}\ltimes \Sp(W)$, $\wideparen{\PGSp^{\pm}}(W)$ denote the corresponding covering groups associated to this cocycle. Then there exists the following commutative diagram:
\begin{equation}\label{eq0}
\begin{CD}
@. 1@. 1 @.1 @. \\
@.@VVV @VVV @VVV @.\\
1 @>>> 1@>>> F^{\times}@>>> F^{\times }@>>> 1\\
@.@VVV @VVV @VVV @.\\
1@>>>\Sp(W) @>>>\wideparen{\widetilde{F^{\times}}}\ltimes \Sp(W)@>\wideparen{\lambda}>>\wideparen{\widetilde{F^{\times}}} @>>> 1\\
@. @\vert @VV\wideparen{p}V @VVV @.\\
1 @>>> \Sp(W) @>>> \wideparen{\PGSp^{\pm}}(W)@>\wideparen{\dot{\lambda}}>> \wideparen{F^{\times}/F^{\times 2}} @>>> 1\\
@.@VVV @VVV @VVV @.\\
@. 1@. 1 @.1 @.
\end{CD}
\end{equation}
Let us consider a twisted right action of $F^{\times} \times \Sp(W)$ on $\Ha(W)$ in the following way:
\begin{equation}\label{alphaac}
\alpha:\Ha(W) \times [F^{\times} \times \Sp(W)]\longrightarrow \Ha(W); ((v,t), [k, g]) \longmapsto (v g, t).
\end{equation}
The restriction of $\alpha$ on $\Sp(W)$ is the usual action. Let us extend this action to $\wideparen{\widetilde{F^{\times}}}\ltimes \Sp(W)$. For any element $\wideparen{h} \in \wideparen{\widetilde{F^{\times}}}\ltimes \Sp(W)$, let us write 
$$ \wideparen{h} =(h, \epsilon),  h=[t_h, g_h],  t_h=a_h \kappa(\dot{t}_h)$$
for some  $h\in \widetilde{F^{\times}}\ltimes \Sp(W), \epsilon \in F^{\times},   t_h\in \wideparen{\widetilde{F^{\times}}}, g_h\in \Sp(W), a_h\in F^{\times}, \dot{t}_h \in F^{\times }/F^{\times 2}$.  Recall that $\lambda: \wideparen{\widetilde{F^{\times}}}\ltimes \Sp(W) \longrightarrow \wideparen{\widetilde{F^{\times}}} \longrightarrow F^{\times}$.
\begin{lemma}
\begin{itemize}
\item[(1)] $a_{h}^{-1}a_{h'}^{-1}=c'(h, h') a_{hh'}^{-1}$.
\item[(2)] $\lambda_{a_{h}^{-1}h}\lambda_{a_{h'}^{-1}h'}=c'(h,h')^2\lambda_{a_{hh'}^{-1}hh'}$.
\end{itemize}
\end{lemma}
\begin{proof}
1) As  $t_{h}t_{h'}=t_{hh'}$,
$$a_h\kappa(\dot{t_h})a_{h'}\kappa(\dot{t_{h'}})=a_ha_{h'}c'(\dot{t_h}, \dot{t_{h'}})\kappa(\dot{t_{hh'}}) =a_{hh'} \kappa(\dot{t_{hh'}}),$$
$$c'(\dot{t_h}, \dot{t_{h'}})=c'(h, h').$$
2) It is a consequence of (1).
\end{proof}
Via the group homomorphism: $\widetilde{F^{\times}}\ltimes \Sp(W) \longrightarrow \GSp(W)$, we can get an action of $\widetilde{F^{\times}}\ltimes \Sp(W)$ on $W$. 
Now  let us extend the twisted  action $\alpha$ to $\wideparen{\widetilde{F^{\times}}}\ltimes \Sp(W)$ as follows:
\begin{equation}\label{alphaac}
\alpha: \Ha(W) \times [\wideparen{\widetilde{F^{\times}}} \ltimes \Sp(W)] \longrightarrow \Ha(W); ((v,t), (h,\epsilon)) \longmapsto (a_{h}^{-1} v h\epsilon, \lambda_{a_{h}^{-1}h} \lambda_{\epsilon}  t).
\end{equation}
By the above lemma, it is well-defined. Moreover, $[F^{\times} \times 1]$  acts trivially on $\Ha(W)$,  so it indeed defines an action of $\wideparen{\PGSp^{\pm}}(W)$ on $\Ha(W)$ by the diagram (\ref{eq0}).

 Via the projections:  $\widetilde{F^{\times}} \ltimes \Mp(W)\longrightarrow \widetilde{F^{\times}}$,  $\widetilde{F^{\times}} \ltimes \overline{\Sp}(W)\longrightarrow \widetilde{F^{\times}}$, $\PGMp^{\pm}(W) \longrightarrow F^{\times}/F^{\times 2}$, $\overline{\PGSp^{\pm}}(W) \longrightarrow F^{\times}/F^{\times 2}$,  we can also view $c'(-,-)$ as a cocycle on $\widetilde{F^{\times}} \ltimes \Mp(W)$, $\widetilde{F^{\times}} \ltimes \overline{\Sp}(W)$,  $\PGMp^{\pm}(W)$, $\overline{\PGSp^{\pm}}(W) $.  Let  $\wideparen{\widetilde{F^{\times}}}\ltimes \Mp(W)$,  $\wideparen{\widetilde{F^{\times}}}\ltimes \overline{\Sp}(W)$, $\wideparen{\PGMp^{\pm}}(W)$,  $\wideparen{\overline{\PGSp^{\pm}}}(W) $  denote the corresponding covering groups associated to $c'(-,-)$.  Then there exist the following commutative diagrams:
\begin{equation}\label{eq5}
\begin{CD}
1@>>>\overline{\Sp}(W) @>>>\wideparen{\overline{\PGSp^{\pm}}}(W)@> \wideparen{\dot{\lambda}}>> \wideparen{F^{\times}/F^{\times 2}} @>>> 1\\
@. @VVV @VV\wideparen{p}V @\vert @.\\
1 @>>> \Sp(W) @>>> \wideparen{\PGSp^{\pm}}(W)@>\wideparen{\dot{\lambda}}>> \wideparen{F^{\times}/F^{\times 2}} @>>> 1\\
\end{CD}
\end{equation}
\begin{equation}\label{eq6}
\begin{CD}
1@>>>\Mp(W) @>>>\wideparen{\PGMp^{\pm}}(W)@> \wideparen{\dot{\lambda}}>> \wideparen{F^{\times}/F^{\times 2}} @>>> 1\\
@. @VVV @VV\wideparen{p}V @\vert @.\\
1 @>>> \Sp(W) @>>> \wideparen{\PGSp^{\pm}}(W)@>\wideparen{\dot{\lambda}}>> \wideparen{F^{\times}/F^{\times 2}} @>>> 1\\
\end{CD}
\end{equation}
\subsection{Extended  Weil representations}
Recall that $(\pi_{\psi}, V_{\psi})$ is a Weil representation of $\Mp(W)$ associated to $\psi$. Let us define the twisted induced Weil representation of $\PGMp^{\pm}(W)$ as follows:
$$\Pi_{\psi}=\cInd_{\Mp(W)}^{\PGMp^{\pm}(W)} \pi_{\psi},\quad \mathcal{V}_{\psi}=\cInd_{\Mp(W)}^{\PGMp^{\pm}(W)} V_{\psi}.$$
By considering the restriction of $\Pi_{\psi}$ on $\Mp(W)$, we can obtain four Weil representations, corresponding  to $\pi_{\psi^a}$, as $a=1, -1, \varpiup,-\varpiup$. So it is independent  of  the choice of $\psi$. This representation can also be viewed as a representation of  $\overline{\PGSp^{\pm}}(W)$. By Clifford-Mackey theory, it is also compatible with the induced Weil representation directly from $\overline{\Sp}(W)$ to $\overline{\PGSp^{\pm}}(W)$.  The next purpose is to see whether such  representations  can arise    from the Heisenberg  representation of $\Ha(W)$.

Recall the definition of the $\alpha$ action  in (\ref{alphaac}).  Through the above projection $\wideparen{p}$, we can extend this  action to  $\wideparen{\PGMp^{\pm}}(W)$. Then let us define a representation as follows:
$$\wideparen{\Pi}_{\psi}=\cInd_{\Mp(W)\ltimes \Ha(W)}^{\wideparen{\PGMp^{\pm}}(W) \ltimes_{\alpha} \Ha(W)}\pi_{\psi},\quad \wideparen{\mathcal{V}}_{\psi}=\cInd_{\Mp(W)\ltimes \Ha(W)}^{\wideparen{\PGMp^{\pm}}(W) \ltimes_{\alpha} \Ha(W)} V_{\psi}.$$
Notice that $\wideparen{\PGMp^{\pm}}(W)$ is a central extension of $\PGMp^{\pm}(W)$ by $F^{\times}$. So there   exists a  question on how to reasonably eliminate the effects of the central group. Similarly, we can define the corresponding representation of $\wideparen{\overline{\PGSp^{\pm}}}(W)\ltimes_{\alpha} \Ha(W)$.
\section{The group $\PGMp^{\pm}(W)$: Case $|k_F|\equiv 1(\bmod4)$}\label{case1mod411}
In the following two sections,  we  assume $|k_F|\equiv 1(\bmod4)$.
\subsection{$\widetilde{F^{\times}}$}\label{wideparen}  Follow  the  notations in Section \ref{Ftimes}. Recall    $F^{\times} \simeq    \mathfrak{f} \times U_1\times  \langle \varpiup\rangle$, and $ \mathfrak{f} =\langle \zeta_1\rangle \times \langle \zeta_2\rangle $. Recall the exact sequence (\ref{Ftwo}):
 \begin{align}
1 \longrightarrow  \{\pm 1\} \longrightarrow  \widetilde{\mathfrak{f}^2}\longrightarrow \mathfrak{f}^{2} \longrightarrow 1,
\end{align}
Recall that $c_{\sqrt{\,}}(-,-)$ is a $2$-cocycle associated to this exact sequence. Moreover, there exists an isomorphism: $\sqrt{\,}: \widetilde{\mathfrak{f}^2} \longrightarrow \mathfrak{f}; [a^2, \epsilon] \to \sqrt{a^2} \epsilon$.  Let us extend this $2$-cocycle to $F^{\times 2}\simeq \mathfrak{f}^2 \times U_1\times  \langle \varpiup^2 \rangle $, and obtain a group $\widetilde{F^{\times 2}}= \widetilde{\mathfrak{f}^2} \times U_1\times  \langle \varpiup^2 \rangle $. Then there exists an isomorphism:
$$\sqrt{\,}: \widetilde{F^{\times 2}} \to F^{\times}; ([a^{2}, \epsilon], u, \varpiup^{2k}) \longmapsto  (\sqrt{a^2} \epsilon, \sqrt{u}, (-\varpiup)^k). $$
Recall the exact sequence:
\begin{equation}\label{fff}
1\longrightarrow F^{\times 2} \longrightarrow F^{\times} \longrightarrow F^{\times}/F^{\times 2} \longrightarrow 1.
\end{equation}
It is known that  $ F^{\times}/F^{\times 2}=\{ \dot{1},  \dot{\zeta_1}, \dot{\varpiup}, \dot{(\zeta_1\varpiup)} \}$. Let $\kappa$ be the canonical section map from $F^{\times}/F^{\times 2}$ to $F^{\times}$ such that $\kappa( \dot{1})=1$, $\kappa(\dot{\zeta_1})=\zeta_1$, $\kappa(\dot{\varpiup})=\varpiup$, $\kappa(\dot{(\zeta_1\varpiup)})=\zeta_1\varpiup$. Let $c'(-,-)$ denote the $2$-cocycle associated to $\kappa$. More precisely, for two elements $t_1, t_2\in F^{\times}$, let us write $t_1=a_1^2 \kappa(\dot{t}_1)$, $t_2=a_2^2 \kappa(\dot{t}_2)$. Then:
  \begin{align}\label{tttt}
 t_1t_2=(a_1a_2)^2 \kappa(\dot{t}_1) \kappa(\dot{t}_2)=(a_1a_2)^2 c'(\dot{t}_1, \dot{t}_2)\kappa(\dot{t}_1\dot{t}_2).
 \end{align}
 Hence $c'(\dot{t}_1, \dot{t}_2)=c'(\dot{t}_2, \dot{t}_1)$, and $c'(\dot{1}, \dot{t}_2)=1$,  $c'(\dot{\zeta_1}, \dot{\varpiup})=1$, $c'(\dot{\zeta_1}, \dot{\zeta_1})=\zeta_1^2=c'(\dot{\zeta_1}, \dot{(\zeta_1\varpiup}))$, $c'(\dot{\varpiup}, \dot{\varpiup})=\varpiup^2=c'(\dot{\varpiup}, \dot{(\zeta_1\varpiup)})$, $c'(\dot{(\zeta_1\varpiup)}, \dot{(\zeta_1\varpiup)})=(\zeta_1\varpiup)^2$.
   Recall that there exists a canonical map: $$\varrho :F^{\times 2} \longrightarrow \widetilde{F^{\times 2}}; a \longmapsto [a,1].$$
Through $\varrho $, we view $c'(-,-)$ as a map from $F^{\times}/F^{\times 2} \times F^{\times}/F^{\times 2}$ to $\widetilde{F^{\times 2}}$. Let us write $c''(-,-)=\varrho (c'(-,-))$.
\begin{lemma}
$c''(-,-)$ defines a $2$-cocycle on $F^{\times}/F^{\times 2}$ with values in $\widetilde{F^{\times 2}}$.
\end{lemma}
\begin{proof}
1) $c''(\dot{1}, \dot{t})=[1,1]=c''(\dot{t}, \dot{1})$.\\
2) $$c''(\dot{a}, \dot{b})c''(\dot{ab}, \dot{c})=[c'(\dot{a}, \dot{b}),1][c'(\dot{ab}, \dot{c}),1]=[c'(\dot{a}, \dot{b})c'(\dot{ab}, \dot{c}),c_{\sqrt{\,}}(c'(\dot{a}, \dot{b}),c'(\dot{ab}, \dot{c}))];$$
 $$c''(\dot{a}, \dot{bc})c''(\dot{b}, \dot{c})=[c'(\dot{a}, \dot{bc}),1][c'(\dot{b}, \dot{c}),1]=[c'(\dot{a}, \dot{bc})c'(\dot{b}, \dot{c}),c_{\sqrt{\,}}(c'(\dot{a}, \dot{bc}),c'(\dot{b}, \dot{c}))].$$
 Note that $c'(\dot{a}, \dot{b})c'(\dot{ab}, \dot{c})=c'(\dot{a}, \dot{bc})c'(\dot{b}, \dot{c})$. So, it suffices to show that
\begin{equation}\label{fffffff}
c_{\sqrt{\,}}(c'(\dot{a}, \dot{b}),c'(\dot{ab}, \dot{c}))=c_{\sqrt{\,}}(c'(\dot{a}, \dot{bc}),c'(\dot{b}, \dot{c})),
\end{equation}
 for  $\dot{a}, \dot{b}, \dot{c} \in F^{\times}/F^{\times 2}$. Let us write  $\mathfrak{a}=c'(\dot{a}, \dot{b})$,  $\mathfrak{b}=c'(\dot{ab}, \dot{c})$,  $\mathfrak{c}=c'(\dot{a}, \dot{bc})$, $\mathfrak{d}=c'(\dot{b}, \dot{c})$. Then:
 $$\mathfrak{a}\mathfrak{b}=\mathfrak{c}\mathfrak{d},  \quad \quad \mathfrak{a}, \mathfrak{b},\mathfrak{c}, \mathfrak{d}\in \{ 1, \zeta_1^2, \varpiup^2, (\zeta_1\varpiup)^2\}.$$
 i) If $\zeta_1^4\neq 1$  and  one of $\mathfrak{a}, \mathfrak{b},\mathfrak{c}, \mathfrak{d}$ equals to $1$, by the symmetry of the cocycle, we assume $\mathfrak{a}=1$. Then: $\mathfrak{b}=\mathfrak{c}\mathfrak{d}$. So $\left\{\begin{array}{l} \mathfrak{b}=\mathfrak{c}\\ \mathfrak{d}=1\end{array}\right.$, $\left\{ \begin{array}{l}\mathfrak{b}=\mathfrak{d}\\ \mathfrak{c}=1 \end{array}\right.$,  $\left\{ \begin{array}{l}\mathfrak{b}=(\zeta_1\varpiup)^2 \\ \mathfrak{c}\neq \mathfrak{d}\in \{\zeta_1^2, \varpiup^2\} \end{array}\right.$. In each case, the equality (\ref{fffffff}) holds true. \\
 ii) If  $\zeta_1^4\neq 1$, and none of $\mathfrak{a}, \mathfrak{b},\mathfrak{c}, \mathfrak{d}$ equals to $1$, then $\mathfrak{a}\neq  \mathfrak{b} $ iff $\mathfrak{c}\neq \mathfrak{d}$. If $\mathfrak{a}= \mathfrak{b}$, then $\mathfrak{c}=  \mathfrak{d}=\mathfrak{a}$, the equality (\ref{fffffff}) holds. If $\mathfrak{a}\neq  \mathfrak{b}$, then $\{\mathfrak{a}, \mathfrak{b}\}=\{\mathfrak{c}, \mathfrak{d}\}$. So the equality (\ref{fffffff}) also holds.\\
 iii) If $\zeta_1^4=1$,  by symmetry,   we only need to add the following possible  cases to (i)(ii):
 $$\left\{ \begin{array}{l} \mathfrak{a}=1=\mathfrak{b} \\ \mathfrak{c}= \mathfrak{d}=\zeta_1^2 \end{array}\right., \quad \left\{ \begin{array}{l} \mathfrak{a}\neq \mathfrak{b}\in \{ 1, \varpiup^2\}\\ \mathfrak{c}\neq \mathfrak{d}\in \{ \varpiup^2 \zeta_1^2 ,\zeta_1^2\} \end{array}\right., \left\{ \begin{array}{l} \mathfrak{a}= \mathfrak{b}=\varpiup^2\\ \mathfrak{c}= \mathfrak{d}=\varpiup^2 \zeta_1^2  \end{array}\right..$$
 (1) If $ c'(\dot{a}, \dot{bc})=c'(\dot{b}, \dot{c})=\zeta_1^2$, then $\dot{b}=\dot{c}=\dot{\zeta_1}$, or $\dot{b}\neq \dot{c} \in \{ \dot{\zeta_1}, \dot{(\zeta_1\varpiup)}\}$; in both cases, $c'(\dot{a}, \dot{bc})$ can't be $\zeta_1^2$.\\
 (2) If $ c'(\dot{a}, \dot{bc})= \varpiup^2 \zeta_1^2$ and $c'(\dot{b}, \dot{c})=\zeta_1^2$, then  $\dot{b}=\dot{c}=\dot{\zeta_1}$, or $\dot{b}\neq \dot{c} \in \{ \dot{\zeta_1}, \dot{(\zeta_1\varpiup)}\}$. In both cases, $ \dot{bc}= \dot{1}$ or $\dot{\varpiup}$, then $c'(\dot{a}, \dot{bc}) \neq \varpiup^2 \zeta_1^2$.\\
 (3) If $ c'(\dot{a}, \dot{bc})= \zeta_1^2$ and $c'(\dot{b}, \dot{c})= \varpiup^2\zeta_1^2$,  then $\dot{b}=\dot{c}=\dot{(\zeta_1\varpiup)}$ and $\dot{bc}=\dot{1}$, it is also impossible.\\
 (4) If $ c'(\dot{a}, \dot{bc})=c'(\dot{b}, \dot{c})=\varpiup^2 \zeta_1^2$, then $\dot{b}=\dot{c}=\dot{(\zeta_1\varpiup)}$ and $\dot{bc}=\dot{1}$, it is also impossible.
\end{proof}
 Associated to the  $2$-cocycle $c''$, there exists an exact sequence:
\begin{equation}\label{fffww}
1\longrightarrow \widetilde{F^{\times 2}} \longrightarrow \widetilde{F^{\times}} \longrightarrow F^{\times}/F^{\times 2} \longrightarrow 1.
\end{equation}
\begin{lemma}
There exists a group homomorphism:
$$ \widetilde{F^{\times }} \longrightarrow F^{\times}; ([g,\epsilon], \dot{t}) \longmapsto [g, \dot{t}]$$ such that
 the following  commutative diagram holds.
\[
\begin{CD}\label{eq78}
 @. 1@. 1 @.1 @. \\
@.@VVV @VVV  @VVV \\
1 @>>> \{\pm 1\}@>>>   \{\pm 1\} @>>>1 @>>> 1\\
@.@VVV @VVV  @VVV \\
1 @>>> \widetilde{F^{\times 2}}@>>>  \widetilde{F^{\times}} @>>>F^{\times}/F^{\times 2} @>>> 1\\
@.@VVV @VVV @VV{=}V \\
1 @>>> F^{\times 2}@>>> F^{\times} @>>>F^{\times}/F^{\times 2} @>>> 1\\
@.@VVV @VVV  @VVV \\
 @. 1@. 1 @.1 @. \\
\end{CD}
\]
\end{lemma}
\begin{proof}
It is clear that the map satisfies the commutative property. So it suffices to show that  the map is also a group homomorphism. For $([g_1, \epsilon_1], \dot{t}_1)$, $([g_1, \epsilon_1], \dot{t}_1) \in  \widetilde{F^{\times}} $,
$$([g_1, \epsilon_1], \dot{t}_1)([g_2, \epsilon_2], \dot{t}_2)=([g_1, \epsilon_1][g_2, \epsilon_2]c''( \dot{t}_1,  \dot{t}_2), \dot{t}_1\dot{t}_2)$$
$$=([g_1, \epsilon_1][g_2, \epsilon_2][c'( \dot{t}_1,  \dot{t}_2), 1], \dot{t}_1\dot{t}_2)=([g_1g_2c'(\dot{t}_1,\dot{t}_2),c_{\sqrt{\ }}(g_1, g_2)c_{\sqrt{\ }}(g_1g_2, c'(\dot{t}_1,\dot{t}_2))], \dot{t}_1\dot{t}_2).$$
On the other hand, in $F^{\times}$:
$$[g_1, \dot{t}_1][g_2, \dot{t}_2]=[g_1g_2c'(\dot{t}_1,\dot{t}_2), \dot{t}_1\dot{t}_2].$$
\end{proof}
 Furthermore, the middle column of the commutative diagram (\ref{eq78}) is associated with a central extension of $F^{\times}$ by $\{\pm 1\}$. Let us denote the corresponding  $2$-cocycle by $c'''(-,-)$. The explicit expression for this cocycle is as follows:
 \begin{align}\label{c33}
c'''(t_1,t_2)=c_{\sqrt{\ }}(a_1^2, a_2^2)c_{\sqrt{\ }}(a_1^2a_2^2, c'(\dot{t}_1,\dot{t}_2)),
 \end{align}
for $t_1=a_1^2 \kappa(\dot{t}_1)$, $t_2=a_2^2 \kappa(\dot{t}_2) \in F^{\times}$.  Note that the group  homomorphism from  $ \widetilde{F^{\times2}}$ to $\widetilde{F^{\times}}$ is given as follows:
$$ \widetilde{F^{\times2}} \hookrightarrow  \widetilde{F^{\times}}; [a^2, \epsilon] \longmapsto [a^2, \epsilon].$$
Recall that $F^{\times} \simeq  \mathfrak{f} \times U_1\times  \langle \varpiup\rangle\simeq   \mathfrak{f}_1 \times \mathfrak{f}_2\times U_1\times  \langle \varpiup\rangle$.
\begin{lemma}\label{1234}
\begin{itemize}
\item[(1)] $c'''(a, b)=1$,  for  $a\in \mathfrak{f}_1$, $b\in \mathfrak{f}_2 \times U_1\times  \langle \varpiup\rangle$.
\item[(2)]  $c'''(a_1b_1,a_2b_2)=c'''(a_1, a_2)$, for  $a_1, a_2\in \mathfrak{f}_1$, $b_1,b_2\in \mathfrak{f}_2 \times U_1\times  \langle \varpiup\rangle$.
\item[(3)] Let $a_1=\zeta_1^{2i_1+l_1}$, $a_2=\zeta_1^{2i_2+l_2}$, for $0\leq i_1, i_2 \leq 2^{k-1}-1$, $0\leq l_1,l_2\leq 1$. Then:
$$c'''(a_1,a_2)=\left\{ \begin{array}{lr}
c_{\sqrt{\ }}(\zeta_1^{2i_1}, \zeta_1^{2i_2})c_{\sqrt{\ }}(\zeta_1^{2i_1+2i_2}, \zeta_1^2) & l_1=l_2=1,\\
c_{\sqrt{\ }}(\zeta_1^{2i_1}, \zeta_1^{2i_2}) &  \textrm{ otherwise.}\end{array}\right.$$
\end{itemize}
\end{lemma}
\begin{proof}
1) Let us write $a=a_1^2 \kappa(\dot{a})$, $b=b_1^2 \kappa(\dot{b})$. By (\ref{c33}),
$$c'''(a,b)=c_{\sqrt{\ }}(a_1^2,b_1^2)c_{\sqrt{\ }}(a_1^2b_1^2, c'(\dot{a},\dot{b})).$$
Then: $c_{\sqrt{\ }}(a_1^2,b_1^2)=1$ and  $\dot{b} \in \{ \dot{1}, \dot{\varpiup}\}$. So $c'(\dot{a},\dot{b})=1$ or $\varpiup^2$. Hence $c_{\sqrt{\ }}(a_1^2b_1^2, c'(\dot{a},\dot{b}))=1$. \\
2) It follows from (1).\\
3) It follows from (\ref{c33}).
\end{proof}
As a consequence,  $\widetilde{F^{\times}} \simeq \widetilde{\mathfrak{f}_1} \times  \mathfrak{f}_2 \times U_1\times  \langle \varpiup\rangle$.
\subsection{$\widetilde{\GSp}(W)$}
Through the projection $\GSp(W) \longrightarrow F^{\times}$, we lift $c'''(-,-)$ from $F^{\times}$ to $\GSp(W)$. Associated to this cocycle, there exists an extension  of $\GSp(W)$ by $\{\pm 1\}$:
$$1 \longrightarrow \{\pm 1\} \longrightarrow \widetilde{\GSp}(W) \longrightarrow \GSp(W) \longrightarrow 1.$$
Composed with $\GSp(W) \longrightarrow F^{\times}/F^{\times 2}$, there  is  a group homomorphism: $ \widetilde{\GSp}(W) \longrightarrow F^{\times}/F^{\times 2}$. Let $\widetilde{F^{\times }\Sp}(W)$ denote its kernel. Then there exists the following commutative diagram:
\[
\begin{CD}\label{eq8}
 @. 1@. 1 @.1 @. \\
@.@VVV @VVV  @VVV \\
1 @>>> \{\pm 1\}@>>>   \{\pm 1\} @>>>1 @>>> 1\\
@.@VVV @VVV  @VVV \\
1 @>>> \widetilde{F^{\times }\Sp}(W)@>>>  \widetilde{\GSp}(W) @>\widetilde{\dot{\lambda}}>>F^{\times}/F^{\times 2} @>>> 1\\
@.@VVV @VVV @VV{=}V \\
1 @>>> F^{\times}\Sp(W)@>>> \GSp(W)  @>\dot{\lambda}>>F^{\times}/F^{\times 2} @>>> 1\\
@.@VVV @VVV  @VVV \\
 @. 1@. 1 @.1 @. \\
\end{CD}
\]
\begin{lemma}
There exists a group homomorphism:
$\widetilde{\lambda}: \widetilde{\GSp}(W)  \longrightarrow \widetilde{F^{\times}}; [g, \epsilon] \longmapsto [\lambda_g, \epsilon]$.
\end{lemma}
\begin{proof}
For $[g_1, \epsilon_1] , [g_2, \epsilon_2] \in \widetilde{\GSp}(W)$, $$[g_1, \epsilon_1] [g_2, \epsilon_2]=[g_1g_2, c'''(\lambda_{g_1},\lambda_{g_2}) \epsilon_1 \epsilon_2].$$
Then:
$$\widetilde{\lambda}([g_1, \epsilon_1] )\widetilde{\lambda}([g_2, \epsilon_2]  )=[\lambda_{g_1}, \epsilon_1][\lambda_{g_2}, \epsilon_2]=[\lambda_{g_1}\lambda_{g_2}, c'''(\lambda_{g_1},\lambda_{g_2}) \epsilon_1 \epsilon_2],$$
$$\widetilde{\lambda}([g_1, \epsilon_1] [g_2, \epsilon_2])=\widetilde{\lambda}([g_1g_2, c'''(\lambda_{g_1},\lambda_{g_2}) \epsilon_1 \epsilon_2])=[\lambda_{g_1g_2}, c'''(\lambda_{g_1},\lambda_{g_2}) \epsilon_1 \epsilon_2].$$
\end{proof}
Note that the projection  $ \widetilde{\GSp}(W) \longrightarrow F^{\times}/F^{\times 2}$,  factors through $\widetilde{\lambda}$.
\begin{lemma}\label{ste}
\begin{itemize}
\item[(1)] $\Sp(W) \simeq \ker(\widetilde{\lambda}); g \longmapsto [g,1]$.
\item[(2)] There exists a group monomorphism: $\sqrt{\,} :\widetilde{F^{\times 2}} \longrightarrow \widetilde{\GSp}(W); [a, \epsilon] \longrightarrow [\sqrt{a}\epsilon, \epsilon]$.   Moreover the image belongs to  the center of  $\widetilde{\GSp}(W)$, and $ \sqrt{\widetilde{F^{\times2}}}\cap \Sp(W)=\{1\}$.
\item[(3)] The above group $\widetilde{F^{\times }\Sp}(W) =\sqrt{\widetilde{F^{\times 2}} }\Sp(W)$.
\end{itemize}
\end{lemma}
\begin{proof}
1) $\ker(\widetilde{\lambda})=\{ [g, 1]\mid g\in \Sp(W)\}$. Moreover, $c'''(\lambda_{g_1}, \lambda_{g_2})=1$, for any $g_1, g_2\in \Sp(W)$. So the result holds.\\
2) For $[a_i, \epsilon_i]\in \widetilde{F^{\times 2}}$,
$$[a_1, \epsilon_1][a_2, \epsilon_2]=[a_1a_2, c_{\sqrt{\ }}(a_1,a_2)\epsilon_1\epsilon_2];$$
$$ [\sqrt{a}_1\epsilon_1, \epsilon_1] [\sqrt{a}_2\epsilon_2, \epsilon_2]=[\sqrt{a}_1\epsilon_1\sqrt{a}_2\epsilon_2, c'''(\lambda_{\sqrt{a}_1\epsilon_1},\lambda_{\sqrt{a}_2\epsilon_2}) \epsilon_1 \epsilon_2]$$
$$=[\sqrt{a}_1\epsilon_1\sqrt{a}_2\epsilon_2, c_{\sqrt{\ }}(a_1,a_2) \epsilon_1 \epsilon_2];$$
$$[\sqrt{a_1a_2}c_{\sqrt{\ }}(a_1,a_2)\epsilon_1\epsilon_2, c_{\sqrt{\ }}(a_1,a_2)\epsilon_1\epsilon_2]=[\sqrt{a}_1\sqrt{a}_2\epsilon_1\epsilon_2,  c_{\sqrt{\ }}(a_1,a_2)\epsilon_1\epsilon_2].$$
If $[1,1]=\sqrt{[a, \epsilon]}=[\sqrt{a}\epsilon, \epsilon]$, then $\sqrt{a}=1=\epsilon$. So $\sqrt{\,}$ is an injective map. For
$[g_1,\epsilon_1]\in \widetilde{\GSp}(W)$, $[\sqrt{a}\epsilon, \epsilon]\in \sqrt{ \widetilde{F^{\times2}}}$,
$$[g_1,\epsilon_1][\sqrt{a}\epsilon, \epsilon]=[g_1\sqrt{a}\epsilon, c'''(\lambda_{g_1},\lambda_{\sqrt{a}\epsilon})\epsilon_1\epsilon]$$
$$=[g_1\sqrt{a}\epsilon, c'''(\lambda_{g_1},a)\epsilon_1\epsilon];$$
$$[\sqrt{a}\epsilon, \epsilon][g_1,\epsilon_1]=[\sqrt{a}\epsilon g_1, c'''(a,\lambda_{g_1})\epsilon_1\epsilon].$$
Since $g_1\sqrt{a}\epsilon=\sqrt{a}\epsilon g_1$ and $ c'''(\lambda_{g_1},a)=c'''(a,\lambda_{g_1})$, $[\sqrt{a}\epsilon, \epsilon]$ lies in the center of $\widetilde{\GSp}(W)$.
If $[\sqrt{a}\epsilon, \epsilon] \in \Sp(W)$,  then $\epsilon=1$, and $\sqrt{a}\epsilon\in \Sp(W)$. Hence $\sqrt{a}=\pm 1$, $a=1$, $\sqrt{a}=1$. \\
3) For $[g, \epsilon]\in \widetilde{\GSp}(W)$, the image of it in $F^{\times}/F^{\times 2}$ equals to $\dot{\lambda}_{g}$. So $\dot{\lambda}_{g}=\dot{1}$ iff $g\in F^{\times}\Sp(W)$. For $\epsilon\in \{\pm 1\}$, and  $tg\in F^{\times} \Sp(W)$ with $t\in F^{\times}$, $g\in \Sp(W)$, we have: 
$$[tg, \epsilon]=[\pm 1 \cdot\sqrt{t^2} g,\epsilon]=[\sqrt{t^2} \epsilon, \epsilon][\pm \epsilon g, 1].$$ Hence $\widetilde{F^{\times }\Sp}(W) \subseteq \sqrt{\widetilde{F^{\times 2}}} \Sp(W)$. It is clear that the other inclusion holds.
\end{proof}
Note that $F^{\times} \hookrightarrow \GSp(W); a\to a$. Let $\widetilde{F^{\times}}^c$ denote the subgroup  of $\widetilde{\GSp}(W)$ over $F^{\times}$.
\begin{lemma}
The center group of $\widetilde{\GSp}(W)$ is just $\widetilde{F^{\times}}^c$.
\end{lemma}
\begin{proof}
Note that the center group contains $\sqrt{\widetilde{F^{\times 2}}} \mu_2=\widetilde{F^{\times}}^c$. The converse is clearly right.
\end{proof}
Note that we can not identity $\widetilde{F^{\times}}^c$ with $\widetilde{F^{\times}}$. However, $\widetilde{F^{\times 2}}$ is just  the covering subgroup  of $\widetilde{\GSp}(W)$ over $F^{\times 2}$.  
\begin{definition}\label{Fposit}
$\widetilde{F^{\times}}_+= \sqrt{\widetilde{F^{\times 2}}}$.
\end{definition}
Then there exist the following exact sequences of groups:
\begin{align}
1 \longrightarrow  \Sp(W) \longrightarrow \widetilde{\GSp}(W) \stackrel{\widetilde{\lambda}}{\longrightarrow}\widetilde{F^{\times}}  \longrightarrow 1;
\end{align}
\begin{align}
1 \longrightarrow \widetilde{F^{\times}}_+\Sp(W) \longrightarrow \widetilde{\GSp}(W) \stackrel{\dot{\widetilde{\lambda}}}{\longrightarrow}\widetilde{F^{\times}} / \widetilde{F^{\times 2}} \simeq F^{\times}/F^{\times 2} \longrightarrow 1.
\end{align}
\subsection{$\widehat{\widetilde{F^{\times}}}$}\label{case1mod412}
Keep the notations of Section \ref{Ftimes}. By Lemma \ref{1234},   $\widetilde{F^{\times}} \simeq \widetilde{\mathfrak{f}_1} \times  \mathfrak{f}_2 \times U_1\times  \langle \varpiup\rangle$. For our purpose,  let us choose a section map $\iota$ from $\widetilde{F^{\times}}$ to $\widetilde{\GSp}(W)$ in the following way:
\begin{align}
\iota:  U_1\longrightarrow  \widetilde{\GSp}(W); &u \longmapsto
(\begin{bmatrix}
\sqrt{u} & 0\\
0 & \sqrt{u} \end{bmatrix},1), \label{tttt1}\\
\iota:  \langle \varpiup\rangle    \longrightarrow  \widetilde{\GSp}(W);& \varpiup \longmapsto (\begin{bmatrix}
0& -I\\
\varpiup I& 0\end{bmatrix},1),\label{tttt2}\\
\iota:  \mathfrak{f}_2     \longrightarrow  \widetilde{\GSp}(W);& x\longmapsto  (\begin{bmatrix}
\sqrt{x} & 0\\
0 & \sqrt{x} \end{bmatrix},1), \label{tttt3}
\end{align}
\begin{align}
\iota: \widetilde{\mathfrak{f}_1}  \longrightarrow  \widetilde{\GSp}(W);&
  \left\{ \begin{array}{lr} (\zeta_1^{2i-1},1)\longmapsto
   ((-\zeta_1)^{i-1}\begin{bmatrix}
0&   -I\\
\zeta_1I& 0\end{bmatrix}, 1) & \textrm{ if }1\leq i\leq 2^{k-1},\\
 (\zeta_1^{2i},1)\longmapsto
  ((-\zeta_1)^{i},1) & \textrm{ if } 1\leq i\leq 2^{k-1}-1,\\
    (1, \epsilon)\longmapsto
   (\epsilon, \epsilon). &
   \end{array}\right.
\end{align}
For $a= \widetilde{a}_1a_2a_3a_4$ with $\widetilde{a}_1\in \widetilde{\mathfrak{f}_1} $, $a_2\in \mathfrak{f}_2$, $a_3\in U_1$,   $a_4\in \langle \varpiup\rangle$, let us define
 $$\iota(a)=\iota(\widetilde{a}_1)\iota(a_2)\iota(a_3)\iota(a_4).$$
 Then   the restriction of $\iota$ on $\mathfrak{f}_2 \times U_1\times \langle \varpiup\rangle $ is a  group homomorphism.
  \begin{lemma}
  \begin{itemize}
  \item[(1)] The restriction of $\iota$ on $\widetilde{\mathfrak{f}^2_1} $ is a group homomorphism, given by $[a^2, \epsilon] \longrightarrow [\sqrt{a^2}\epsilon, \epsilon]$.
  \item[(2)] The restriction of $\iota$ on $\widetilde{F^{\times 2}} $ is a group homomorphism, given by $[a^2, \epsilon] \longrightarrow [\sqrt{a^2}\epsilon, \epsilon]$.
  \item[(3)] The restriction of $\iota$ on $\widetilde{\mathfrak{f}_1} $ is a group homomorphism.
  \end{itemize}
  \end{lemma}
  \begin{proof}
  (1) It suffices to show that it is a group homomorphism.  Note that   $$[a_1^2, \epsilon_1] [a_2^2, \epsilon_2]=[a_1^2a_2^2, c_{\sqrt{\,}}(a_1^2,a_2^2) \epsilon_1\epsilon_2].$$
  So $$\iota([a_1^2, \epsilon_1])\iota([a_2^2, \epsilon_2])=[\sqrt{a_1^2}\epsilon_1, \epsilon_1][\sqrt{a_2^2}\epsilon_2, \epsilon_2]=[\sqrt{a_1^2}\epsilon_1\sqrt{a_2^2}\epsilon_2, \epsilon_1\epsilon_2 c_{\sqrt{\,}}(a_1^2,a_2^2)].$$
  $$\iota([a_1^2a_2^2, c_{\sqrt{\,}}(a_1^2,a_2^2) \epsilon_1\epsilon_2])=[\sqrt{a_1^2a_2^2} c_{\sqrt{\,}}(a_1^2,a_2^2)\epsilon_1\epsilon_2, c_{\sqrt{\,}}(a_1^2,a_2^2)\epsilon_1\epsilon_2]=[\sqrt{a_1^2}\sqrt{a_2^2}\epsilon_1\epsilon_2, c_{\sqrt{\,}}(a_1^2,a_2^2)\epsilon_1\epsilon_2 ].$$
  (2) It follows from (1) and Lemma \ref{1234} and  (\ref{tttt1}),(\ref{tttt2}),(\ref{tttt3}). \\
  (3)(a)   Let us write  $\widetilde{a}_1=(g_1, 1)$, $ \widetilde{a}_2=(g_2, 1)\in \widetilde{\mathfrak{f}_1}$,  $g_j=\zeta_1^{2i_j+l_j} \in F^{\times}$, $0\leq i_1,i_2\leq 2^{k-1}-1$, $0\leq l_1,l_2\leq 1$. Then:
  $$\widetilde{a}_1\widetilde{a}_2=[\zeta_1^{2(i_1+i_2)+l_1+l_2}, c'''(g_1,g_2)], $$
  for $c'''(g_1,g_2)=\left\{ \begin{array}{lr}
c_{\sqrt{\ }}(\zeta_1^{2i_1}, \zeta_1^{2i_2})c_{\sqrt{\ }}(\zeta_1^{2i_1+2i_2}, \zeta_1^2) & l_1=l_2=1,\\
c_{\sqrt{\ }}(\zeta_1^{2i_1}, \zeta_1^{2i_2}) &  \textrm{ otherwise.}\end{array}\right.$  \\
i) If $0\leq l_1\neq l_2\leq 1$, $c'''(g_1,g_2)=c_{\sqrt{\ }}(\zeta_1^{2i_1}, \zeta_1^{2i_2})$.
$$\iota(\widetilde{a}_1\widetilde{a}_2)=[\frac{(-\zeta_1)^{i_1}(-\zeta_1)^{i_2}}{ c_{\sqrt{\ }}(\zeta_1^{2i_1}, \zeta_1^{2i_2})} \begin{bmatrix}
0&   -I\\
\zeta_1I& 0\end{bmatrix}, 1][c'''(g_1,g_2),c'''(g_1,g_2)]=[(-\zeta_1)^{i_1}(-\zeta_1)^{i_2} \begin{bmatrix}
0&   -I\\
\zeta_1I& 0\end{bmatrix},  c'''(g_1,g_2)].$$
j)  If $l_1= l_2=0$, $c'''(g_1,g_2)=c_{\sqrt{\ }}(\zeta_1^{2i_1}, \zeta_1^{2i_2})$.
$$\iota(\widetilde{a}_1\widetilde{a}_2)=[\frac{(-\zeta_1)^{i_1}(-\zeta_1)^{i_2}}{ c_{\sqrt{\ }}(\zeta_1^{2i_1}, \zeta_1^{2i_2})}, 1][c'''(g_1,g_2),c'''(g_1,g_2)]=[(-\zeta_1)^{i_1}(-\zeta_1)^{i_2},  c'''(g_1,g_2)].$$
k)  If $l_1= l_2=1$, $c'''(g_1,g_2)=c_{\sqrt{\ }}(\zeta_1^{2i_1}, \zeta_1^{2i_2})c_{\sqrt{\ }}(\zeta_1^{2i_1+2i_2}, \zeta_1^2) $.
$$\iota(\widetilde{a}_1\widetilde{a}_2)=[\frac{(-\zeta_1)^{i_1}(-\zeta_1)^{i_2}(-\zeta_1)^{1}}{ c_{\sqrt{\ }}(\zeta_1^{2i_1}, \zeta_1^{2i_2})c_{\sqrt{\ }}(\zeta_1^{2i_1+2i_2}, \zeta_1^2)}, 1][c'''(g_1,g_2),c'''(g_1,g_2)]=[(-\zeta_1)^{i_1}(-\zeta_1)^{i_2}(-\zeta_1)^{1},  c'''(g_1,g_2)].$$
 (3)(b) $$\iota(\widetilde{a}_1)=[(-\zeta_1)^{i_1}\begin{bmatrix}
0&   -I\\
\zeta_1I& 0\end{bmatrix}^{l_1}, 1], \quad\quad \iota(\widetilde{a}_1)=[(-\zeta_1)^{i_1}\begin{bmatrix}
0&   -I\\
\zeta_1I& 0\end{bmatrix}^{l_2}, 1];$$
$$\iota(\widetilde{a}_1)\iota(\widetilde{a}_2)=[(-\zeta_1)^{i_1+i_2}\begin{bmatrix}
0&   -I\\
\zeta_1I& 0\end{bmatrix}^{l_1+l_2}, c'''(g_1, g_2)].$$
By the above $(i)(j)(k)$, we know that $\iota(\widetilde{a}_1)\iota(\widetilde{a}_2)=\iota(\widetilde{a}_1\widetilde{a}_2)$.
  \end{proof}
However,  $\iota$ is not a group homomorphism on the whole group $\widetilde{F^{\times}}$. Let $\widehat{\widetilde{F^{\times}}}$ denote the group generated by these $\iota(\widetilde{F^{\times}})$ in $\widetilde{\GSp}(W)$. Let us define
\begin{align} D=\widehat{\widetilde{F^{\times}}} \cap \Sp(W).
\end{align}
\begin{lemma}\label{mod11}
 $D=\langle \begin{bmatrix}
\varpiup^{-1} \xi_1& 0\\
 0& \varpiup\xi_1^{-1}\end{bmatrix}\rangle.$
\end{lemma}
\begin{proof}
Note that  $\iota(\widetilde{F^{\times 2}})$ belongs to the center of $\widetilde{\GSp}(W)$. Moreover, $\widetilde{F^{\times}}/\widetilde{F^{\times 2}}\simeq F^{\times}/F^{\times 2}\simeq \{ \dot{1}, \dot{\zeta}_1, \dot{\varpiup}, \dot{(\zeta_1\varpiup)}\}$. Hence $\widetilde{F^{\times}}= \widetilde{F^{\times 2}} \cup \zeta_1\widetilde{F^{\times 2}}\cup  \varpiup \widetilde{F^{\times 2}}\cup \zeta_1\varpiup \widetilde{F^{\times 2}}$.  By definition, $\iota(\zeta_1\varpiup)=\iota(\zeta_1)\iota(\varpiup)$. By Lemma \ref{1234}, we have: 
$$\iota(\varpiup)\iota(\zeta_1)=(\begin{bmatrix}
0& -I\\
\varpiup I& 0\end{bmatrix}, 1)( \begin{bmatrix}
0&-I\\
 \xi_1 I& 0\end{bmatrix}, 1)=(\begin{bmatrix}
0& -I\\
\varpiup I& 0\end{bmatrix} \begin{bmatrix}
0&-I\\
 \xi_1 I& 0\end{bmatrix}, 1)$$
 $$=(\begin{bmatrix}
0& -I\\
 \xi_1I& 0\end{bmatrix}\begin{bmatrix}
0& -I\\
\varpiup I& 0\end{bmatrix} \begin{bmatrix}
\varpiup^{-1} \xi_1& 0\\
 0& \varpiup\xi_1^{-1}\end{bmatrix}, 1)=\iota(\zeta_1)\iota(\varpiup) ( \begin{bmatrix}
\varpiup^{-1} \xi_1& 0\\
 0& \varpiup\xi_1^{-1}\end{bmatrix}, 1).$$
\end{proof}

Then there exists an  exact sequence:
\begin{equation}\label{exact2mod41}
1 \longrightarrow D \longrightarrow\widehat{\widetilde{F^{\times}}}\stackrel{\lambda}{\longrightarrow}  \widetilde{F^{\times}}  \longrightarrow 1.
\end{equation}
Recall that $\widetilde{F^{\times}}^c$ is the subgroup of $\widetilde{\GSp}(W)$ over $F^{\times}$. Note that there exists a group homomorphism:
$$\iota=\sqrt{\,}: \widetilde{F^{\times 2}} \longrightarrow \widetilde{F^{\times }}^c\subseteq \widetilde{\GSp}(W); [a^2, \epsilon] \longmapsto [\sqrt{a^2} \epsilon, \epsilon].$$

Therefore, every element of $\widehat{\widetilde{F^{\times}}}$ can be written in the form
$$d_1 t_1 \iota(x)$$
for $ d_1\in D, x\in \{ 1,  \zeta_1, \varpiup, (\zeta_1\varpiup) \}, t_1\in \widetilde{F^{\times}}_+$.
 Let $D^2=\{ x^2\mid x\in D\}$, which is a normal subgroup of $\widehat{\widetilde{F^{\times}}}$. Recall the non-abelian group of order $8$:
\begin{align}
D_4=\langle a, b\mid a^4=b^2=1, ba=a^{-1}b\rangle.
\end{align}
\begin{lemma}\label{shex}
There exists a short exact sequence: $1\longrightarrow (\widetilde{F^{\times}}_+) D^2 \longrightarrow \widehat{\widetilde{F^{\times}}} \longrightarrow  D_4\longrightarrow 1$.
\end{lemma}
\begin{proof}
$\iota(\zeta_1\varpiup)=\iota(\zeta_1)\iota(\varpiup)=(\begin{bmatrix} -\varpiup I  & 0 \\ 0 & -\zeta_1 I\end{bmatrix}, 1)$.  Let $a$, $b$ denote the images of $\iota(\zeta_1\varpiup)$ and $\iota(\zeta_1)$ in $\widehat{\widetilde{F^{\times}}}/(\widetilde{F^{\times}}_+) D^2 $ respectively. Then:
$$b^2=\iota(\zeta_1)^2=(-\zeta_1,1)\in \widetilde{F^{\times}}_+;$$
$$ \iota(\zeta_1\varpiup)^2=(\begin{bmatrix} \varpiup^2 I& 0 \\ 0&   \zeta_1^2I\end{bmatrix}, 1)=(\begin{bmatrix} \varpiup \zeta_1^{-1} I& 0 \\ 0&  \zeta_1\varpiup^{-1} I\end{bmatrix}  (\zeta_1 \varpiup ),1)=(\begin{bmatrix} \varpiup \zeta_1^{-1} I& 0 \\ 0&  \zeta_1\varpiup^{-1} I\end{bmatrix}  ,1)( (- \zeta_1)(- \varpiup) ,1);$$
$$ a^4=\iota(\zeta_1\varpiup)^4 \in (\widetilde{F^{\times}}_+) D^2 ; $$
 $$\iota(\zeta_1\varpiup)^3= (\begin{bmatrix} \varpiup^2 I& 0 \\ 0&   \zeta_1^2I\end{bmatrix}, 1)(\begin{bmatrix} -\varpiup I& 0 \\ 0&   -\zeta_1I\end{bmatrix}, 1)$$
 $$=(\begin{bmatrix} -\varpiup^3 I& 0 \\ 0&   -\zeta_1^3I\end{bmatrix}, c'''(\zeta_1^2\varpiup^2,\zeta_1\varpiup))=(\begin{bmatrix} -\varpiup^3 I& 0 \\ 0&   -\zeta_1^3I\end{bmatrix},1)$$
 $$=(\begin{bmatrix} (\zeta_1^{-1}\varpiup)^2 I& 0 \\ 0&   (\zeta_1\varpiup^{-1})^{2} I\end{bmatrix} (\zeta_1 \varpiup )\begin{bmatrix}  -\zeta_1 I& 0 \\ 0&   -\varpiup I\end{bmatrix},1)$$
 $$=(\begin{bmatrix} (\zeta_1^{-1}\varpiup)^2 I& 0 \\ 0&   (\zeta_1\varpiup^{-1})^{2} I\end{bmatrix},1)(\zeta_1 \varpiup ,1)(\begin{bmatrix}  -\zeta_1 I& 0 \\ 0&   -\varpiup I\end{bmatrix},1);$$
 $$\iota(\zeta_1)\iota(\zeta_1\varpiup)=(\begin{bmatrix}  0&  \zeta_1 I \\ -\zeta_1 \varpiup I & 0  \end{bmatrix},1)=(\begin{bmatrix}  -\zeta_1 I& 0 \\ 0&   -\varpiup I\end{bmatrix},1)(\begin{bmatrix}  0&  -I \\ \zeta_1 I & 0  \end{bmatrix},1)=(\begin{bmatrix}  -\zeta_1 I& 0 \\ 0&   -\varpiup I\end{bmatrix},1)\iota(\zeta_1).$$
 Hence  the group $\widetilde{F^{\times}}/(\widetilde{F^{\times}}_+) D^2$ can be generated by $a, b$, and $a^4=1=b^2$, $ba=a^{-1}b$.
\end{proof}
Note that the center group $Z(D_4)$ of $D_4$ is just $\{ 1, a^2\}$.  Hence there exists an exact sequence:
$$1\longrightarrow Z(D_4) \longrightarrow D_4 \longrightarrow \Z_2\times \Z_2 \longrightarrow 1.$$
It corresponds to the following exact sequence:
$$1\longrightarrow D/D^2 \longrightarrow  \frac{\widehat{\widetilde{F^{\times}}}}{ (\widetilde{F^{\times}}_+) D^2} \longrightarrow  F^{\times}/F^{\times 2}\longrightarrow 1.$$

We remark that $D\subseteq M$. For any $y\in \widehat{\widetilde{F^{\times}}}$, $g\in \Sp(W)\subset \widetilde{\GSp}(W)$, let  us define $g^{y}=y^{-1}gy$. Under this conjugacy action, we obtain a semi-product group $\widehat{\widetilde{F^{\times}}} \ltimes \Sp(W)$, and an  exact sequence:
$$1 \longrightarrow \Delta_{D} \longrightarrow \widehat{\widetilde{F^{\times}}}\ltimes \Sp(W) \longrightarrow \widetilde{\GSp}(W) \longrightarrow 1,$$
where $\Delta_{D}=\{ (y, y^{-1})\mid y\in D\}$.
\subsubsection{ $8$-covering groups} We can transfer the action to $\Mp(W)$ and obtain a group $\widehat{\widetilde{F^{\times}}} \ltimes \Mp(W)$, resulting in an exact sequence:
$$ 1\longrightarrow \mu_8 \longrightarrow \widehat{\widetilde{F^{\times}}} \ltimes \Mp(W) \stackrel{\widetilde{p}}{\longrightarrow} \widehat{\widetilde{F^{\times}}} \ltimes \Sp(W) \longrightarrow 1.$$
Let $\widetilde{C}_M(-,-)$ denote the $2$-cocycle   associated to this exact sequence. For  $([y_1, g_1], \epsilon_1), ([y_2, g_2], \epsilon_2)\in \widehat{\widetilde{F^{\times}}}  \ltimes \Mp(W)$, $g_i\in \Sp(W)$, $y_i \in \widehat{\widetilde{F^{\times}}} $, $\epsilon_i\in \mu_8$,
\begin{equation}\label{ygmod41}
\begin{split}
(y_1, g_1, \epsilon_1)\cdot (y_2, g_2, \epsilon_2)&=(y_1y_2, [g_1, \epsilon_1]^{y_2}[g_2, \epsilon_2])\\
& =(y_1y_2, [g_1^{y_2}, \nu(y_2, g_1)\epsilon_1][g_2, \epsilon_2])\\
&=(y_1y_2, [g_1^{y_2}g_2, \nu(y_2, g_1)c_{PR, X^{\ast}}(g_1^{y_2}, g_2)\epsilon_1\epsilon_2]).
\end{split}
\end{equation}
Hence:
\begin{equation}\label{equationsqmod1}
\widetilde{C}_M([y_1,g_1], [y_2, g_2]) =\nu( y_2, g_1)c_{PR, X^{\ast}}(g_1^{y_2}, g_2),   \quad\quad [y_i, g_i]\in \widehat{\widetilde{F^{\times}}} \ltimes \Sp(W).
\end{equation}
\subsubsection{$\nu(-,-)$} Let us find out  the explicit expression of $\nu(-,-)$.
\begin{itemize}
 \item If $y\in D$,  by Example \ref{example1},  $\nu(y,g)=c_{PR, X^{\ast}}(y^{-1}, g y)c_{PR, X^{\ast}}(g,y)=1$.
 \item If $x=zy$, for some $z\in \widehat{\widetilde{F^{\times}}}$, and $y\in D$,  then  by Lemma \ref{twoau},   $\nu(x,g)=\nu(z,g)\nu(y,g^z)=\nu(z,g)$.
 \item If $x=yz=z(z^{-1}yz)$, for some $z\in \widehat{\widetilde{F^{\times}}}$, and $y\in D$,  then  by Lemma \ref{twoau},   $\nu(x,g)=\nu(z,g)\nu( z^{-1}yz,g^{z})=\nu(z,g)$.
 \item If $y\in \iota(\widetilde{F^{\times 2}})$,  then $g^y=g$,  so $\nu(y,-)$ is a character of $\Sp(W)$, and then $\nu(y,g)\equiv 1$.
 \item If $x=zy$,   for some $z\in \widehat{\widetilde{F^{\times}}}$, and $y\in  \iota(\widetilde{F^{\times 2}})$,  then  $\nu(x,g)=\nu(z,g)\nu(y,g^z)=\nu(z,g)$.
 \item If $x=yz$,   for some $z\in \widehat{\widetilde{F^{\times}}}$, and $y\in  \iota(\widetilde{F^{\times 2}})$,  then $g^{y}=g$,  and $\nu(x,g)=\nu(y,g)\nu(z,g^y)=\nu(z,g^y)=\nu(z,g)$.
 \item $\nu(-, -)$ is a function on $(\widetilde{F^{\times }}/\widetilde{F^{\times 2}})\times \Sp(W)\simeq (F^{\times }/F^{\times 2})\times \Sp(W)$.
 \end{itemize}
 Note that   $F^{\times}/F^{\times 2}$ can be represented by $\{ \dot{1}, \dot{\zeta_1}, \dot{\varpiup}, \dot{(\zeta_1\varpiup)}]\}$.
 \begin{lemma}\label{nu1mod41}
 $\nu(y,g)=\left\{ \begin{array}{ll}
(\det a_1a_2, \varpiup)_F \gamma(\varpiup,\psi^{\frac{1}{2}})^{-|S|}c_{PR, X^{\ast}}(\omega^{-1}, g^{x_1} \omega)c_{PR, X^{\ast}}(g,\omega)& \textrm{ if } y=\iota(\varpiup)=x_1\omega,\\
(\det a_1a_2, \zeta_1)_F \gamma(\zeta_1,\psi^{\frac{1}{2}})^{-|S|}c_{PR, X^{\ast}}(\omega^{-1}, g^{y_1} \omega)c_{PR, X^{\ast}}(g,\omega) &\textrm{ if } y=\iota(\zeta_1)=y_1\omega,\\
(\det a_1a_2, \zeta_1\varpiup)_F \gamma(\zeta_1\varpiup, \psi^{\frac{1}{2}})^{-|S|} &\textrm{ if } y=\iota(\zeta_1\varpiup )=\iota(\zeta_1)\iota(\varpiup),
\end{array}\right.$
 for $g=p_1\omega_S p_2$, $p_i=\begin{pmatrix}
a_i& b_i\\
0& d_i
\end{pmatrix}$, $i=1,2$.
\end{lemma}
\begin{proof}
 If $y=\iota(\varpiup)$, let us write $y=x_1\omega$, for $x_1= \begin{bmatrix}
1& \\
& \varpiup\end{bmatrix}$. By Examples \ref{example1},\ref{example2} and Lemma \ref{twoau},
 $$\nu(y,g)=\nu(x_1,g)\nu(\omega, g^{x_1})=(\det a_1a_2, \varpiup)_F \gamma(\varpiup, \psi^{\frac{1}{2}})^{-|S|}c_{PR, X^{\ast}}(\omega^{-1}, g^{x_1} \omega)c_{PR, X^{\ast}}(g^{x_1},\omega).$$ Note that $c_{PR, X^{\ast}}(g^{x_1},\omega)=c_{PR, X^{\ast}}(g,\omega)$. The expression  for  $\iota(\zeta_1) $ is similar. If  $y=\iota(\zeta_1\varpiup)$, then $$y=\begin{bmatrix}
-\varpiup& \\
& -\zeta_1\end{bmatrix}=z_1h, \quad \quad \textrm{ for } z_1=\begin{bmatrix}
1& \\
& \zeta_1\varpiup\end{bmatrix}, h=\begin{bmatrix}
-\varpiup& \\
& -\varpiup^{-1}\end{bmatrix}.$$ By Examples \ref{example1},\ref{example2},
\begin{align*}
&\nu(y,g)=\nu(z_1,g)\nu(h, g^{z_1})\\
&=(\det a_1a_2, \zeta_1\varpiup)_F \gamma(\zeta_1\varpiup, \psi^{\frac{1}{2}})^{-|S|}c_{PR, X^{\ast}}(h^{-1}, g^{z_1}h)c_{PR, X^{\ast}}(g,h)\\
&=(\det a_1a_2, \zeta_1\varpiup)_F \gamma(\zeta_1\varpiup, \psi^{\frac{1}{2}})^{-|S|}.
\end{align*}
\end{proof}

\begin{lemma}
$\nu(\widetilde{y},\omega_S)=\gamma(\lambda_{y},\psi^{\frac{1}{2}})^{-|S|}$, for $\widetilde{y}=(y, \epsilon)\in \widehat{\widetilde{F^{\times}}}\subseteq \widetilde{\GSp}(W)$.
\end{lemma}
\begin{proof}
If $\dot{y}=\dot{1} \in F^{\times}/F^{\times2}$, the result holds. If $y=\iota(\varpiup)=x_1\omega$,  $$\nu(y,\omega_S)= \gamma(\varpiup,\psi^{\frac{1}{2}})^{-|S|}c_{PR, X^{\ast}}(\omega^{-1}, \omega_S^{x_1} \omega)c_{PR, X^{\ast}}(\omega_S,\omega)=\gamma(\varpiup,\psi^{\frac{1}{2}})^{-|S|}.$$
It is similar for the other elements.
\end{proof}
\begin{lemma}\label{gg0mod41}
For $g_0=\begin{pmatrix}
a& 0\\
0& (a^{\ast})^{-1}
\end{pmatrix}\in M$, $g\in \Sp(W)$, $\widetilde{y}=(y, \epsilon)\in \widehat{\widetilde{F^{\times}}}$, $\lambda_y\in F^{\times}$,
\begin{itemize}
\item[(1)] $\nu(y, g_0)=(\det a,\lambda_y)_F$,
\item[(2)] $\nu(y, g_0g)=\nu(y, g_0) \nu(y, g)$,
\item[(3)] $\nu(y, gg_0)= \nu(y, g)\nu(y, g_0)$.
\end{itemize}
\end{lemma}
\begin{proof}
1) By the formulas in Lemma \ref{nu1mod41}, the result holds.  \\
2) Let us write $g=p_1\omega_S p_2$.   Then $g_0g=g_0p_1\omega_S p_2$. So the result is also right for $y=\iota(\zeta_1\varpiup)$.   For $y=\iota(\varpiup)$, or $\iota(\zeta_1)$,
we can use a similar proof as in Lemma \ref{gg0mod43}.\\
3) The proof is similar to that of Lemma \ref{gg0mod43}(3). 
\end{proof}

\subsubsection{$\widetilde{\GMp}(W)$}
Recall the composite map $\lambda: \widetilde{\GSp}(W) \stackrel{\widetilde{\lambda}}{\longrightarrow } \widetilde{F^{\times}} \longrightarrow F^{\times}$.
\begin{lemma}\label{coco1204}
For $[y_1, g_1]\in \Delta_D$, with $g_1=y_1^{-1}= \begin{bmatrix}
t& \\
& t^{-1}\end{bmatrix}$,  $[y_2, g_2]\in \widehat{\widetilde{F^{\times}}} \ltimes \Sp(W)$, we have:
$$\widetilde{C}_M([y_1,g_1], [y_2, g_2]) =( t^m, \lambda_{y_2})_F, \textrm{ and } \widetilde{C}_M([y_2,g_2], [y_1, g_1]) = 1.$$
\end{lemma}
\begin{proof}
By (\ref{equationsqmod1}), $$\widetilde{C}_M([y_1,g_1], [y_2, g_2])=\nu( y_2, g_1)c_{PR, X^{\ast}}(g_1^{y_2}, g_2)=\nu(y_2, g_1)=( t^m, \lambda_{y_2})_F,$$
 $$\widetilde{C}_M([y_2,g_2], [y_1, g_1])=\nu( y_1, g_2   )c_{PR, X^{\ast}}(g_2^{y_1}, g_1)=\nu( y_1, g_2   )=1.$$
\end{proof}
\begin{corollary}
$\widetilde{C}_M([y_2,g_2], [y_1, g_1]) = 1$, for $[y_i, g_i]\in \Delta_D$.
\end{corollary}
Hence, there is an exact sequence:
\begin{align}\label{mu8mod41}
1 \longrightarrow \mu_8\times \Delta_D \longrightarrow \widehat{\widetilde{F^{\times}}} \ltimes \Mp(W) \stackrel{p_1}{\longrightarrow} \widetilde{\GSp}(W) \longrightarrow 1.
\end{align}

\begin{lemma}\label{zzz2}
Let $Z$ denote the center group  of $\widehat{\widetilde{F^{\times}}} \ltimes \Mp(W)$. Then:
  $Z=[\widetilde{F^{\times}}_+ \times \{\pm 1\} ]\times \mu_8$.
\end{lemma}
\begin{proof}
For any $[y_1, g_1]\in \widetilde{F^{\times}}_+ \times 1$, $[y_2, g_2]\in \widehat{\widetilde{F^{\times}}} \ltimes \Sp(W)$,
 $$\widetilde{C}_M([y_1,g_1], [y_2, g_2])=\nu( y_2, g_1)c_{PR, X^{\ast}}(g_1^{y_2}, g_2)=\nu( y_2, g_1)c_{PR, X^{\ast}}(g_1, g_2)=1;$$
$$\widetilde{C}_M([y_2,g_2], [y_1, g_1])=\nu( y_1, g_2)c_{PR, X^{\ast}}(g_2^{y_1}, g_1)=\nu( y_1, g_2)=1.$$
For  the other possible  central element $[y_1, -1]$, $y_1\in \widetilde{F^{\times}}_+$,
$$\widetilde{C}_M([y_2,g_2], [y_1, -1])=1, \quad \widetilde{C}_M([y_1,-1], [y_2, g_2])=(\lambda_{y_2}, (-1)^m)_F=1.$$
So the result  holds.
\end{proof}
Note that the above exact sequence (\ref{mu8mod41}) means that $\widetilde{F^{\times}} \ltimes \Mp(W)$ is a \emph{ normal } extension of $\GSp(W)$ by $\mu_8\times \Delta_D$. So let us modify the $2$-cocycle $\widetilde{C}_M$. By identifying $D$ with $\Delta_D$, we view $ \widetilde{C}_M(-, -)$ as a function of  $\widetilde{F^{\times}}_+D\times \widehat{\widetilde{F^{\times}}}  $ and $\widehat{\widetilde{F^{\times}}}  \times \widetilde{F^{\times}}_+ D$, denoted by $\widetilde{C}'_M(-, -)$.
 \begin{lemma}
 \begin{itemize}
 \item[(1)] If $2\mid m$, $\widetilde{C}'_M|_{[D\widetilde{F^{\times}}_+]  \times \widehat{\widetilde{F^{\times}}} }=1=\widetilde{C}'_M|_{  \widehat{\widetilde{F^{\times}}}\times [D\widetilde{F^{\times}}_+ ]  }$.
 \item[(2)] If $2\nmid m$, $\widetilde{C}'_M|_{[ D^2\widetilde{F^{\times}}_+]\times  \widehat{\widetilde{F^{\times}}} }=1=\widetilde{C}'_M|_{ \widehat{\widetilde{F^{\times}}} \times [D^2\widetilde{F^{\times}}_+]  }$.
 \end{itemize}
 \end{lemma}
 \begin{proof}
 It follows from the above lemmas \ref{coco1204}, \ref{zzz2}.
 \end{proof}
 Next, we will extend $\widetilde{C}'_M$ to be a cocycle on $\widehat{\widetilde{F^{\times}}} $, denoted by $$c_{\widehat{\widetilde{F^{\times}}}}(-,-).$$
 \begin{itemize}
\item If  $2\textbar m$, we can let  $c_{\widehat{\widetilde{F^{\times}}}}(x, y) \equiv 1$. 
\item If  $2\nmid m$, following Lemma \ref{shex}, we can view $\widetilde{C}'_M$ as a function on $\langle a^2 \rangle \times D_4$ and  $ D_4\times \langle a^2 \rangle$ such that 
$$\left\{ \begin{array}{l}
\widetilde{C}'_M(a^2, 1)=\widetilde{C}'_M(a^2, a)=\widetilde{C}'_M(a^2, a^2)=\widetilde{C}'_M(a^2, a^3)=1,\\
\widetilde{C}'_M(a^2, b)=\widetilde{C}'_M(a^2, ab)=\widetilde{C}'_M(a^2, a^2b)=\widetilde{C}'_M(a^2, a^3b)=-1,\\
\widetilde{C}'_M(y,a^2)=1, \textrm{ for any } y\in  D_4.
\end{array}\right.$$
\begin{lemma}\label{CM111}
$\widetilde{C}'_M$ cannot be extended to yield a cocycle on  $D_4$ with values in  $\mu_2$.
\end{lemma}
\begin{proof}
Assuming that such cocycle  $c_{\widehat{\widetilde{F^{\times}}}}$ exists,  we can construct an exact sequence 
$1\longrightarrow \mu_2 \longrightarrow G \longrightarrow D_4 \to 1.$ The center group $Z(G)$ is then considered, and its image in $D_4$ cannot be $\langle a^2\rangle $ due to the aforementioned relations. Therefore, $Z(G)=\mu_2$. According to \cite{Ll}, there are only three possible types of groups that $G$ can be.  Additionally, each of these groups contains a subgroup isomorphic to $\Z_8$, whose image in $D_4$ much be $\langle a\rangle$. The cocycle on $\langle a\rangle$ is determined by $c_{\widehat{\widetilde{F^{\times}}}}(a,a)$. If $c_{\widehat{\widetilde{F^{\times}}}}(a,a)=1$, then $c_{\widehat{\widetilde{F^{\times}}}}$ is a trivial cocycle on $\langle a\rangle$,  which contradicts the fact that $Z(G)$ is isomorphic to $\Z_8$.  On the other hand, if $c_{\widehat{\widetilde{F^{\times}}}}(a,a)=-1$, then the cover group over $\langle a \rangle $ can not be $\Z_8$. 
\end{proof}
For the above reason, let us construct a cocycle  $c_{\widehat{\widetilde{F^{\times}}}}$  on  $D_4$ with values in  $\mu_4$:\\
 Write $D_4=\{ a^{k}, ba^{l}\mid 0\leq k, l\leq 3\}$, and define 
$$\left\{ \begin{array}{l}
c_{\widehat{\widetilde{F^{\times}}}}(a^l, a^k)=c_{\widehat{\widetilde{F^{\times}}}}(ba^l, a^k)=1,\\
c_{\widehat{\widetilde{F^{\times}}}}(a^l, ba^k)=c_{\widehat{\widetilde{F^{\times}}}}(ba^l, ba^k)=e^{\tfrac{2\pi i l}{4}}.
\end{array}\right.$$
It is clear that $c_{\widehat{\widetilde{F^{\times}}}}$ extends the function $\widetilde{C}'_M$. 
\begin{lemma}\label{CM112}
$c_{\widehat{\widetilde{F^{\times}}}}$ is a $2$-cocycle on $D_4$.
\end{lemma}
\begin{proof}
1) It is clear that $c_{\widehat{\widetilde{F^{\times}}}}(1,x)=c_{\widehat{\widetilde{F^{\times}}}}(x,1)=1$.\\
2) Let us check that $c_{\widehat{\widetilde{F^{\times}}}}(x,y)c_{\widehat{\widetilde{F^{\times}}}}(xy,z)=c_{\widehat{\widetilde{F^{\times}}}}(x,yz)c_{\widehat{\widetilde{F^{\times}}}}(y,z)$.

(i) If $z=a^j, y=a^k$,  then the left hand is equal to the right hand, resulting in one.

(iia)  If $z=a^j, y=ba^k$, $x=a^l$,  then the left hand is equal to the right hand, resulting in $e^{\tfrac{2\pi i l}{4}}$.

(iib) If $z=a^j, y=ba^k$, $x=ba^l$,  then the left hand is equal to the right hand, resulting in $e^{\tfrac{2\pi i l}{4}}$.

(iiia) If $z=ba^j, y=a^k$, $x=a^l$,  then the left hand is equal to the right hand, resulting in $e^{\tfrac{2\pi i (k+l)}{4}}$.

(iiib) If $z=ba^j, y=a^k$, $x=ba^l$,  then the left hand is equal to the right hand, resulting in $e^{\tfrac{2\pi i (k+l)}{4}}$.

(iva) If $z=ba^j, y=ba^k$, $x=a^l$,  then the left hand is equal to the right hand, resulting in $e^{\tfrac{2\pi i k}{4}}$.

(ivb) If $z=ba^j, y=ba^k$, $x=ba^l$,  then the left hand is equal to the right hand, resulting in $e^{\tfrac{2\pi i k}{4}}$.
\end{proof}
\end{itemize} 
 
Via the projection $ \widehat{\widetilde{F^{\times}}} \ltimes  \Sp(W)\longrightarrow  \widehat{\widetilde{F^{\times}}}$, we view $c_{\widehat{\widetilde{F^{\times}}}}(-, -)$ as a $2$-cocycle on $\widehat{\widetilde{F^{\times}}} \ltimes \Sp(W) $. For our purpose, let us modify the above $2$-cocycle $\widetilde{C}_M(-,-)$ as follows:
\begin{align}\label{doublCM}
\widetilde{\widetilde{C}}_M([y_1,g_1], [y_2, g_2])=\widetilde{C}_M ([y_1,g_1], [y_2, g_2]) c_{\widehat{\widetilde{F^{\times}}}}(y_1, y_2)^{-1},
\end{align}
for $ [y_i, g_i]\in \widehat{\widetilde{F^{\times}}} \ltimes \Sp(W)$.
\begin{lemma}
$\widetilde{\widetilde{C}}_M([y_1,g_1], [y_2, g_2])=1=\widetilde{\widetilde{C}}_M([y_2,g_2], [y_1, g_1])$, for $[y_1, g_1]\in \Delta_D$, $y_1=g_1^{-1}=\begin{pmatrix}
t& 0\\
0 & t^{-1}\end{pmatrix}$.
\end{lemma}
\begin{proof}
$$\widetilde{\widetilde{C}}_M([y_1,g_1], [y_2, g_2])=\nu( y_2, g_1)c_{PR, X^{\ast}}(g_1^{y_2}, g_2)c_{\widehat{\widetilde{F^{\times}}}}(y_1, y_2)^{-1}$$
$$=\nu( y_2, g_1)c_{\widehat{\widetilde{F^{\times}}}}(y_1, y_2)^{-1}=(t^m, \lambda_{y_2})_F (t^m, \lambda_{y_2})_F^{-1}=1.$$
$$\widetilde{\widetilde{C}}_M([y_2,g_2], [y_1, g_1])=\nu( y_1, g_2)c_{PR, X^{\ast}}(g_2^{y_1}, g_1)c_{\widehat{\widetilde{F^{\times}}}}(y_2, y_1)^{-1}=c_{\widehat{\widetilde{F^{\times}}}}(y_2, y_1)^{-1}=1.$$
\end{proof}
As a consequence, for any $\widetilde{g}\in \Delta_D$, $\widetilde{g}_1, \widetilde{g}_2\in \widehat{\widetilde{F^{\times}}} \ltimes \Sp(W)$,
$$\widetilde{\widetilde{C}}_M(\widetilde{g}\widetilde{g}_1, \widetilde{g}_2)=\widetilde{\widetilde{C}}_M(\widetilde{g}, \widetilde{g}_1)\widetilde{\widetilde{C}}_M(\widetilde{g}\widetilde{g}_1, \widetilde{g}_2)=\widetilde{\widetilde{C}}_M(\widetilde{g}, \widetilde{g}_1 \widetilde{g}_2)\widetilde{\widetilde{C}}_M(\widetilde{g}_1, \widetilde{g}_2)=\widetilde{\widetilde{C}}_M(\widetilde{g}_1, \widetilde{g}_2).$$
Similarly, $\widetilde{\widetilde{C}}_M(\widetilde{g}_1, \widetilde{g}_2\widetilde{g})=\widetilde{\widetilde{C}}_M(\widetilde{g}_1, \widetilde{g}_2)$.
 Let $  \widehat{\widetilde{F^{\times}}}\ltimes \widetilde{\Sp}(W)$ denote the corresponding  central covering group.  Then there exists the following commutative diagram:
\begin{equation}\label{eq9}
\begin{CD}
@. 1@. 1 @.1 @. \\
@.@VVV @VVV @VVV @.\\
1 @>>> 1 @>>> \mu_8 @= \mu_8 @>>> 1\\
@.@VVV @VVV @VVV @.\\
1 @>>> \Delta_D @>>>  \widehat{\widetilde{F^{\times}}} \ltimes \widetilde{\Sp}(W) @>p_M>> \widetilde{\GMp}(W)@>>> 1\\
@. @\vert @VV\widetilde{p}V @VVpV @.\\
1@>>>\Delta_D @>>>\widehat{\widetilde{F^{\times}}} \ltimes \Sp(W)@>p_S>>\widetilde{\GSp}(W) @>>> 1\\
@.@VVV @VVV @VVV @.\\
@. 1@. 1 @.1 @.
\end{CD}
\end{equation}
where
\begin{align}
 p_M:          &\widehat{\widetilde{F^{\times}}} \ltimes \widetilde{\Sp}(W)\longrightarrow  \widetilde{\GMp}(W);(y, [g,\epsilon]) \longmapsto (yg, \epsilon), \\
 p_S:          &\widehat{\widetilde{F^{\times}}} \ltimes \Sp(W)\longrightarrow  \widetilde{\GSp}(W);(y, g) \longmapsto yg, \\
 \widetilde{p}:& \widehat{\widetilde{F^{\times}}} \ltimes  \widetilde{\Sp}(W)\longrightarrow  \widehat{\widetilde{F^{\times}}} \ltimes \Sp(W);(y, [g,\epsilon]) \longmapsto (y, g),\\
 p:  & \widetilde{\GMp}(W)\longrightarrow  \widetilde{\GSp}(W);(h,\epsilon) \longmapsto h.
 \end{align}
Let us denote the $2$-cocycle associated to $\widetilde{\GMp}(W)$  by $C_M(-,-)$.   Let $(y_1, [g_1, \epsilon_1]), (y_2, [g_2, \epsilon_2])\in \widehat{\widetilde{F^{\times}}}\ltimes \widetilde{\Sp}(W)$. Then:
$$(y_1, [g_1, \epsilon_1])\cdot (y_2, [g_2, \epsilon_2])=(y_1y_2, g_1^{y_2}g_2, \widetilde{\widetilde{C}}_M([y_1,g_1], [y_2, g_2])\epsilon_1\epsilon_2]).$$
Note that $y_1y_2 g_1^{y_2}g_2=y_1g_1y_2g_2$. So
$$C_M(y_1g_1, y_2g_2)=\widetilde{\widetilde{C}}_M([y_1,g_1], [y_2, g_2])=\nu(y_2, g_1)c_{PR, X^{\ast}}(g_1^{y_2}, g_2)c_{\widehat{\widetilde{F^{\times}}}}(y_1, y_2)^{-1}.$$
Note that the restriction of $C_M(-, -)$ on $\Sp(W)$ is just the Perrin-Rao's cocycle.
\begin{lemma}\label{cen3}
The center of $\widetilde{\GMp}(W)$ contains the subgroup  $\widetilde{F^{\times}}_+ \times \mu_8$.
\end{lemma}
\begin{proof}
For any $g\in \widetilde{\GSp}(W)$, let us write $g=y_2g_2$, with $y_2\in \widehat{\widetilde{F^{\times}}}$, $g_2\in \Sp(W)$, and for  $t\in \widetilde{F^{\times}}_+$, $t=t\cdot 1_{\Sp(W)}$. Then:
$$C_M(t, g)=\nu(y_2, 1_{\Sp(W)})c_{PR, X^{\ast}}(1_{\Sp(W)}^{y_2}, g_2)c_{\widehat{\widetilde{F^{\times}}}}(t, y_2)^{-1}=c_{\widehat{\widetilde{F^{\times}}}}(t, y_2)^{-1}=1,$$
$$C_M(g,t)=\nu(t, g_2)c_{PR, X^{\ast}}(g_2^{t}, 1_{\Sp(W)})c_{\widehat{\widetilde{F^{\times}}}}(y_2, t)^{-1}=c_{\widehat{\widetilde{F^{\times}}}}(y_2, t)^{-1}=1.$$
\end{proof}
\subsection{The group $\overline{\widetilde{\GSp}}(W)$}\label{case1mod413}
We can also  lift the action of $\widehat{\widetilde{F^{\times}}}$ from  $\Sp(W)$ to $\overline{\Sp}(W)$, and then  obtain a group $\widehat{\widetilde{F^{\times}}}\ltimes \overline{\Sp}(W)$ as well as an exact sequence:
$$ 1\longrightarrow \mu_2 \longrightarrow \widehat{\widetilde{F^{\times}}} \ltimes \overline{\Sp}(W) \stackrel{\widetilde{p}}{\longrightarrow} \widehat{\widetilde{F^{\times}}} \ltimes \Sp(W) \longrightarrow 1.$$
Let $\overline{C}_M(-,-)$ denote the $2$-cocycle   associated to this exact sequence. Similarly,
\begin{equation}\label{equations2}
\overline{C}_M([y_1,g_1], [y_2, g_2]) =\nu_2( y_2, g_1)c(g_1^{y_2}, g_2),   \quad\quad [y_i, g_i]\in \widehat{\widetilde{F^{\times}}} \ltimes \Sp(W).
\end{equation}

Let us also give the explicit expression of $\nu_2(y_2, g_1)$.  For $y\in \widehat{\widetilde{F^{\times}}} $, $g\in \Sp(W)$, according to Lemma \ref{nu2},
$$\nu_2( y, g)=\nu(y,g) \frac{m_{X^{\ast}}(g)}{m_{X^{\ast}}(g^{y})}.$$
\begin{lemma}
 $\nu_2(-,-)$ is a function on $(\widehat{\widetilde{F^{\times}}} /[(\widetilde{F^{\times}}_+) D^2]) \times \Sp(W)$.
\end{lemma}
\begin{proof}
For $y\in D^2$ or $y\in  \widetilde{F^{\times}}_+$, by (\ref{mx}), $m_{X^{\ast}}(g^y)=m_{X^{\ast}}(g)$. Hence $\nu_2( y, g)=\nu(y,g)=1$, for such $y$.
 If $x=zy$, for some $z\in  \widehat{\widetilde{F^{\times}}} $, and $y\in ( \widetilde{F^{\times}}_+) D^2$,  then  by Lemma \ref{twoau},   $\nu_2(x,g)=\nu_2(z,g)\nu_2(y,g^z)=\nu_2(z,g)$.
If $x=yz=z(z^{-1}yz)$, for some $z\in  \widehat{\widetilde{F^{\times}}} $, and $y\in ( \widetilde{F^{\times}}_+) D^2$,  then  by Lemma \ref{twoau},   $\nu_2(x,g)=\nu_2(z,g)\nu_2( z^{-1}yz,g^{z})=\nu_2(z,g)$.
\end{proof}
Note that by Lemma \ref{shex}, $\widehat{\widetilde{F^{\times}}} /[(\widetilde{F^{\times}}_+) D^2]  \simeq D_4$.
\begin{lemma}
For $g_0=\begin{pmatrix}
a& 0\\
0& (a^{\ast})^{-1}
\end{pmatrix}\in M$, $g\in \Sp(W)$, $y\in  \widehat{\widetilde{F^{\times}}} $, $\lambda_y\in F^{\times}$, we have:
\begin{itemize}
\item[(1)] $\nu_2(y, g_0)=(\det a,\lambda_y)_F$,
\item[(2)] $\nu_2(y, g_0g)=\nu_2(y, g_0) \nu_2(y, g)(\det a, x(g)x(g^y))_F$,
\item[(3)] $\nu_2(y, gg_0)= \nu_2(y, g)\nu_2(y, g_0)(\det a, x(g)x(g^y))_F$.
\end{itemize}
\end{lemma}
\begin{proof}
1) By (\ref{mx}), $m_{X^{\ast}}(g_0^y)=m_{X^{\ast}}(g_0)$, so $ \nu_2(y, g_0)=\nu(y, g_0)=(\det a,\lambda_y)_F$.\\
2)  By (\ref{mx}), $m_{X^{\ast}}(g_0g)=m_{X^{\ast}}(g_0)m_{X^{\ast}}(g)(\det a, x(g))_F$. Since $g_0^y \in M$,
\begin{align*}
m_{X^{\ast}}(g_0^yg^y)&=m_{X^{\ast}}(g_0^y)m_{X^{\ast}}(g^y)(\det a, x(g^y))_F\\
&=m_{X^{\ast}}(g_0)m_{X^{\ast}}(g^y)(\det a, x(g^y))_F.
\end{align*}
Hence:
\begin{align*}
\nu_2(y, g_0g)&=\nu(y,g_0g) \frac{m_{X^{\ast}}(g_0g)}{m_{X^{\ast}}(g_0^yg^{y})}\\
&=\nu(y, g_0) \nu(y, g) \frac{m_{X^{\ast}}(g_0)m_{X^{\ast}}(g)(\det a, x(g))_F}{m_{X^{\ast}}(g_0^y)m_{X^{\ast}}(g^y)(\det a, x(g^y))_F}\\
&=\nu_2(y, g_0) \nu_2(y, g)(\det a, x(g)x(g^y))_F.
\end{align*}
3) The proof is similar as (2).
\end{proof}
\begin{lemma}
\begin{align}\label{comp}
\overline{C}_M([y_1,g_1], [y_2, g_2])=\widetilde{C}_M([y_1,g_1], [y_2, g_2])m_{X^{\ast}}(g_1^{y_2}g_2)^{-1}m_{X^{\ast}}(g_1) m_{X^{\ast}}(g_2).
\end{align}
\end{lemma}
\begin{proof}
By Lemma \ref{nu2}, $\nu_2(y_2, g_1)=\nu(y_2,g_1) \frac{m_{X^{\ast}}(g)}{m_{X^{\ast}}(g^{y_2})}$. By (\ref{28inter}), $$c(g_1^{y_2}, g_2)=m_{X^{\ast}}(g_1^{y_2}g_2)^{-1} m_{X^{\ast}}(g_1^{y_2}) m_{X^{\ast}}(g_2) c_{PR, {X^{\ast}}}(g_1^{y_2},g_2).$$
Hence:
\begin{align*}
\overline{C}_M([y_1,g_1], [y_2, g_2])&=\nu(y_2,g_1) \frac{m_{X^{\ast}}(g_1)}{m_{X^{\ast}}(g_1^{y_2})}m_{X^{\ast}}(g_1^{y_2}g_2)^{-1} m_{X^{\ast}}(g_1^{y_2}) m_{X^{\ast}}(g_2) c_{PR, {X^{\ast}}}(g_1^{y_2},g_2)\\
&=\widetilde{C}_M([y_1,g_1], [y_2, g_2])m_{X^{\ast}}(g_1^{y_2}g_2)^{-1}m_{X^{\ast}}(g_1) m_{X^{\ast}}(g_2).
\end{align*}
\end{proof}
\begin{lemma}\label{coco1}
For $[y_1, g_1]\in \Delta_D$,  $g_1=y_1^{-1}= \begin{bmatrix}
t& \\
& t^{-1}\end{bmatrix}$,  $[y_2, g_2]\in  \widehat{\widetilde{F^{\times}}}  \ltimes \Sp(W)$, $g_2=p_1\omega_S p_2$, $p_i=\begin{pmatrix}
a_i& b_i\\
0& d_i
\end{pmatrix}$, $i=1,2$,
$$\overline{C}_M([y_1,g_1], [y_2, g_2])=( t^m, \lambda_{y_2})_F(t^m, \det a_1a_2)_F, \quad \quad \overline{C}_M([y_2,g_2], [y_1, g_1]) =(t^m, \det a_1a_2)_F.$$
\end{lemma}
\begin{proof}
By (\ref{comp}),
\begin{align*}
&\overline{C}_M([y_1,g_1], [y_2, g_2])\\
& =\widetilde{C}_M([y_1,g_1], [y_2, g_2])m_{X^{\ast}}(g_1^{y_2}g_2)^{-1}m_{X^{\ast}}(g_1) m_{X^{\ast}}(g_2)\\
&=\widetilde{C}_M([y_1,g_1], [y_2, g_2])m_{X^{\ast}}(g_1^{y_2})^{-1}m_{X^{\ast}}(g_2)^{-1}(t^m, x(g_2))_Fm_{X^{\ast}}(g_1) m_{X^{\ast}}(g_2)\\
&= ( t^m, \lambda_{y_2})_F(t^m, x(g_2))_F\\
&=( t^m, \lambda_{y_2})_F(t^m, \det a_1a_2)_F.
\end{align*}
\begin{align*}
\overline{C}_M([y_2,g_2], [y_1, g_1])& =\widetilde{C}_M([y_2,g_2], [y_1, g_1])m_{X^{\ast}}(g_2^{y_1}g_1)^{-1}m_{X^{\ast}}(g_1) m_{X^{\ast}}(g_2)\\
&=\widetilde{C}_M([y_2,g_2], [y_1, g_1])m_{X^{\ast}}(g_1g_2)^{-1}m_{X^{\ast}}(g_1) m_{X^{\ast}}(g_2)\\
&=(t^m, x(g_2))_F=(t^m, \det a_1a_2)_F.
\end{align*}
\end{proof}
Similarly,  let us modify the above $2$-cocycle $\overline{C}_M(-,-)$ as follows:
\begin{align}\label{doublCM21}
\overline{\overline{C}}_M([y_1,g_1], [y_2, g_2])=\overline{C}_M ([y_1,g_1], [y_2, g_2]) c_{ \widehat{\widetilde{F^{\times}}} }(y_1, y_2)^{-1},
\end{align}
for $ [y_i, g_i]\in  \widehat{\widetilde{F^{\times}}}  \ltimes \Sp(W)$.
\begin{lemma}
$\overline{\overline{C}}_M([y_1,g_1], [y_2, g_2])=\widetilde{\widetilde{C}}_M([y_1,g_1], [y_2, g_2])m_{X^{\ast}}(g_1^{y_2}g_2)^{-1}m_{X^{\ast}}(g_1) m_{X^{\ast}}(g_2)$.
\end{lemma}
\begin{proof}
It follows from Lemma \ref{comp}, and (\ref{doublCM})(\ref{doublCM21}).
\end{proof}
\begin{lemma}\label{equalD}
$\overline{\overline{C}}_M([y_1,g_1], [y_2, g_2])=\overline{\overline{C}}_M([y_2,g_2], [y_1, g_1])$, for $[y_1, g_1]\in \Delta_D$, $y_1=g_1^{-1}=\begin{pmatrix}
t& 0\\
0 & t^{-1}\end{pmatrix}$.
\end{lemma}
\begin{proof}
$$\overline{\overline{C}}_M([y_1,g_1], [y_2, g_2])=( t^m, \lambda_{y_2})_F(t^m, x(g_2))_Fc_{ \widehat{\widetilde{F^{\times}}} }(y_1, y_2)^{-1}=(t^m, x(g_2))_F.$$
$$\overline{\overline{C}}_M([y_2,g_2], [y_1, g_1])=(t^m,  x(g_2))_Fc_{ \widehat{\widetilde{F^{\times}}} }(y_2, y_1)^{-1}=(t^m, x(g_2))_F.$$
\end{proof}
\begin{corollary}
$\overline{\overline{C}}_M([y_1,g_1], [y_2, g_2])=1=\overline{\overline{C}}_M([y_2,g_2], [y_1, g_1])$, for $[y_1, g_1], [y_2,g_2]\in \Delta_D$.
\end{corollary}
\begin{proof}
 If $2\mid  m$, by the above lemma,  the result holds. Similarly, if $2\nmid m$, it follows from the fact that  $(\varpiup \xi_1^{-1},\varpiup \xi_1^{-1})_F=(\varpiup \xi_1^{-1},-1)_F=1$.
\end{proof}
 Let  $\overline{\widehat{\widetilde{F^{\times}}}\ltimes\Sp}(W)$ denote the corresponding  central covering group associated to $\overline{\overline{C}}_M(-,-)$.
\begin{lemma}
 $\Delta_D$ is a normal subgroup of    $\overline{\widehat{\widetilde{F^{\times}}}\ltimes\Sp}(W)$.
\end{lemma}
\begin{proof}
By the above lemma, $\overline{\overline{C}}_M([y_1,g_1], [y_2, g_2])=(t^m, x(g_2))_F=\overline{\overline{C}}_M([y_2,g_2], [y_1, g_1])$, for any $[y_1, g_1]\in \Delta_D$, $y_1=g_1^{-1}=\begin{pmatrix}
t& 0\\
0 & t^{-1}\end{pmatrix}$. Note that $y_1^{y_2}=y_2^{-1} y_1y_2 =y_1$ or $y_1^{-1}$, and $(t, a)_F=(t^{-1}, a)_F$, for any $a\in F^{\times}$. Then:
\begin{align*}
([y_1,g_1], 1)([y_2, g_2], \epsilon_2)&=([y_1,g_1][y_2, g_2], (t^m, x(g_2))_F\epsilon_2)\\
&=([y_1y_2,g_1^{y_2}g_2], (t^m, x(g_2))_F\epsilon_2);
\end{align*}
\begin{align*}
&([y_2, g_2], \epsilon_2)([y_1^{y_2},g_1^{y_2}], 1)\\
&=([y_2, g_2][y_1^{y_2},g_1^{y_2}],(t^m, x(g_2))_F \epsilon_2)\\
&=([y_2y_2^{-1}y_1y_2,(y_2^{-1}y_1^{-1}y_2g_2y_2^{-1}y_1y_2)(y_2^{-1}g_1y_2)],(t^m, x(g_2))_F \epsilon_2)\\
&=([y_1y_2,y_2^{-1}y_1^{-1}y_2g_2y_2^{-1}y_1g_1y_2],(t^m, x(g_2))_F \epsilon_2)\\
&=([y_1y_2,y_2^{-1}g_1y_2g_2],(t^m, x(g_2))_F \epsilon_2)\\
&=([y_1y_2,g_1^{y_2}g_2],(t^m, x(g_2))_F \epsilon_2).
\end{align*}
Hence
$$([y_2, g_2], \epsilon_2)^{-1}([y_1,g_1], 1)([y_2, g_2], \epsilon_2)=([y_1^{y_2},g_1^{y_2}], 1).$$
\end{proof}
 Then:
  \begin{equation}\label{eq10}
\begin{CD}
@. 1@. 1 @.1 @. \\
@.@VVV @VVV @VVV @.\\
1 @>>> 1 @>>> \mu_4 @= \mu_4 @>>> 1\\
@.@VVV @VVV @VVV @.\\
1 @>>> \Delta_D @>>> \overline{ \widehat{\widetilde{F^{\times}}}  \ltimes \Sp}(W)@>\overline{p_S}>> \overline{\widetilde{\GSp}}(W)@>>> 1\\
@. @\vert @VV\widetilde{p}V @VVpV @.\\
1@>>>\Delta_D @>>>\widehat{\widetilde{F^{\times}}}  \ltimes \Sp(W)@>p_S>>\widetilde{\GSp}(W) @>>> 1\\
@.@VVV @VVV @VVV @.\\
@. 1@. 1 @.1 @.
\end{CD}
\end{equation}
In the penultimate column,  it is the following exact sequence:
\begin{align}\label{mumod4122}
1 \longrightarrow \mu_4\longrightarrow \overline{\widetilde{\GSp}}(W) \longrightarrow \widetilde{\GSp}(W)  \longrightarrow 1.
\end{align}
 Let us denote the $2$-cocycle associated to $\overline{\widetilde{\GSp}}(W)$  by $\overline{C}(-,-)$.   For $h_1, h_2\in \widetilde{\GSp}(W)$, let us write
 \begin{align}\label{zzcoc4122}
 h_i=y_ig_i,
 \end{align}
 for $y_i=\iota(\widetilde{\lambda}_{h_i})\in \widehat{\widetilde{F^{\times}}} $, $g_i=y_i^{-1} h_i\in \Sp(W)$.\footnote{Note that in Lemma \ref{equalD}, the value of the $2$-cocycle is not necessarily $1$. Therefore, $\overline{C}(-,-)$ depends on the choice made in $(\ref{zzcoc4122})$.} 
 So we can choose the $\overline{C}(-,-)$ associated to $\overline{\widetilde{\GSp}}(W)$ as follows:
$$\overline{C}(h_1, h_2)= \overline{\overline{C}}_M([y_1,g_1], [y_2, g_2])=\nu_2( y_2, g_1)c(g_1^{y_2}, g_2) c_{\widetilde{F^{\times}}}(y_1, y_2)^{-1}.$$
Then:
\begin{align}\label{zzcoc114122}
\overline{C}(h_1, h_2)=C_M(h_1, h_2)m_{X^{\ast}}(g_1^{y_2}g_2)^{-1}m_{X^{\ast}}(g_1) m_{X^{\ast}}(g_2).
\end{align}

 \begin{lemma}\label{cen41}
 \begin{itemize}
 \item[(1)] $\overline{C}(h_1, h_2)=1=\overline{C}(h_2, h_1)$, for $h_1\in \widetilde{F^{\times}}_+$.
 \item[(2)] The center of $\overline{\widetilde{\GSp}}(W)$ contains the subgroup  $\widetilde{F^{\times}}_+ \times \mu_4$.
 \end{itemize}
\end{lemma}
\begin{proof}
For $h_i$, let us write $h_i=y_i g_i$ by the formula (\ref{zzcoc4122}). Then for $h_1=t\in \widetilde{F^{\times}}_+$, $\widetilde{\lambda}_{h_1}\in \widetilde{F^{\times 2}}$ and  $\iota(\widetilde{\lambda}_{h_1})=t$. Hence $y_1=t$, $g_1=1_{\Sp(W)}$.\\
1) By Lemma \ref{cen3}, $C_M(h_1, h_2)=C_M(h_2, h_1)=1$.  As $y_1\in  \widetilde{F^{\times}}_+$, $g_2^{y_1}=g_2$. Then:
$$ m_{X^{\ast}}(g_1^{y_2}g_2)^{-1}m_{X^{\ast}}(g_1) m_{X^{\ast}}(g_2)= m_{X^{\ast}}(g_2)^{-1}m_{X^{\ast}}(1_{\Sp(W)}) m_{X^{\ast}}(g_2)=1;$$
$$m_{X^{\ast}}(g_2^{y_1}g_1)^{-1}m_{X^{\ast}}(g_2) m_{X^{\ast}}(g_1)=m_{X^{\ast}}(g_2g_1)^{-1}m_{X^{\ast}}(g_2) m_{X^{\ast}}(g_1)=1.$$
By (\ref{zzcoc114122}), the result holds.\\
2) It follows from (1).
\end{proof}
\subsubsection{} 
Let us  define: 
\begin{itemize}
\item $\PGSp^{\pm }(W)=\widetilde{\GSp}(W)/\widetilde{F^{\times}}_+$
\item $\overline{\PGSp^{\pm}}(W)=\overline{\widetilde{\GSp}}(W)/\widetilde{F^{\times}}_+$
\item $\PGMp^{\pm}(W)=\widetilde{\GMp}(W)/{\widetilde{F^{\times}}_+}$.
\end{itemize}
Then there exists the following exact sequence:
\begin{align}
1 \longrightarrow \Sp(W) \longrightarrow \PGSp^{\pm }(W) \stackrel{\dot{\widetilde{\lambda}}}{\longrightarrow}\widetilde{F^{\times}} / \widetilde{F^{\times 2}} \simeq F^{\times}/F^{\times 2} \longrightarrow 1.
\end{align}
By   Lemmas  \ref{cen3}, \ref{cen41}, the two $2$-cocycles $C_M(-,-)$, $\overline{C}(-,-)$ are  defined on $\PGSp^{\pm}(W)$. Then $\PGMp^{\pm}(W)$ and  $\overline{\PGSp^{\pm}}(W)$ are central extensions of  $\PGSp^{\pm}(W)$ by $\mu_8$ and  $\mu_4$, respectively, associated with  $C_M(-,-)$ and  $\overline{C}(-,-)$. Moreover, there exist two exact sequences:
$$1 \to \Mp(W) \to \PGMp^{\pm}(W) \stackrel{\dot{\lambda}}{\longrightarrow} F^{\times}/F^{\times 2} \to 1.$$
$$1 \to \overline{\Sp}(W) \to \overline{\PGSp^{\pm}}(W)\stackrel{\dot{\lambda}}{\longrightarrow} F^{\times}/F^{\times 2} \to 1.$$
\section{Extended  Weil representation: Case $|k_F|\equiv 1(\bmod4)$}\label{case1mod414}
\subsection{} Recall $\widetilde{F^{\times}} \simeq    \widetilde{\mathfrak{f}_1} \times   \mathfrak{f}_2 \times U_1\times  \langle \varpiup\rangle$.  There exists the following exact sequence(cf. (\ref{fffww})):
\begin{equation}\label{fff12mod41}
1\longrightarrow \widetilde{F^{\times 2}} \longrightarrow \widetilde{F^{\times}} \longrightarrow \widetilde{F^{\times}} /\widetilde{F^{\times 2}} \simeq F^{\times}/F^{\times 2} \longrightarrow 1.
\end{equation}
Note that $\sqrt{\widetilde{F^{\times 2}}}=\widetilde{F^{\times}}_+$. Recall that $c''(-,-)$ is a $2$-cocycle associated to this exact sequence. The image of $c''(-,-)$ belongs to $  \widetilde{\mathfrak{f}_1^2}\times  \langle \varpiup^2\rangle$.  Let us define
$$\wideparen{c''}(-,-)=\sqrt{c''(-,-)}.$$
 Via the projections: $\widetilde{F^{\times} }\longrightarrow \widetilde{F^{\times} }/\widetilde{F^{\times 2}}$, $\widetilde{\GSp}(W) \longrightarrow \widetilde{F^{\times}}$, $\PGSp^{\pm}(W) \stackrel{\widetilde{\dot{\lambda}}}{\longrightarrow} F^{\times}/F^{\times 2}$, we can also view $\wideparen{c''}(-,-)$ as a cocycle on $\widetilde{F^{\times} }$,$\widetilde{\GSp}(W)$, or $\PGSp^{\pm}(W)$.  Let $\wideparen{\widetilde{F^{\times} }/\widetilde{F^{\times 2}}}$, $\wideparen{\widetilde{F^{\times} }}$, $\wideparen{\widetilde{\GSp}}(W)$,  $\wideparen{\PGSp^{\pm}}(W)$ denote the corresponding covering groups associated to $\wideparen{c''}(-,-)$. Then there exists the following commutative diagram:
\begin{equation}\label{eq11}
\begin{CD}
@. 1@. 1 @.1 @. \\
@.@VVV @VVV @VVV @.\\
1 @>>> 1@>>> \widetilde{F^{\times}}_+@>\{- \}^2>> \widetilde{F^{\times 2}}@>>> 1\\
@.@VVV @VVV @VVV @.\\
1@>>>\Sp(W) @>>>\wideparen{\widetilde{\GSp}}(W)@>\wideparen{\lambda}>>\wideparen{\widetilde{F^{\times} }} @>>> 1\\
@. @\vert @VV\wideparen{p}V @VVV @.\\
1 @>>> \Sp(W) @>>> \wideparen{\PGSp^{\pm}}(W)@>\wideparen{\dot{\widetilde{\lambda}}}>> \wideparen{F^{\times}/F^{\times 2}} @>>> 1\\
@.@VVV @VVV @VVV @.\\
@. 1@. 1 @.1 @.
\end{CD}
\end{equation}
\subsection{A twisted  action}  Recall that there exists  a group homomorphism  $\widetilde{\GSp}(W) \longrightarrow \GSp(W)$. So there exists an action of $\widetilde{\GSp}(W)$ on $\Ha(W)$. Recall $\lambda: \widetilde{\GSp}(W) \stackrel{\widetilde{\lambda}}{\longrightarrow} \widetilde{F^{\times}} \longrightarrow F^{\times}$. Recall that $\widetilde{F^{\times 2}}=\widetilde{\mathfrak{f}^2} \times U_1\times  \langle \varpiup^2\rangle$, and there exists a group homomorphism:
 \begin{align}
\sqrt{\,}: \widetilde{F^{\times 2}} \longrightarrow  \widetilde{F^{\times}}_+ \subseteq \widetilde{\GSp}(W).
\end{align}
Moreover, $\widetilde{\lambda}_{\sqrt{(a^2, \epsilon)}}=(a^2, \epsilon)$.     Let us consider a twisted right action of $\widetilde{F^{\times}}_+\Sp(W)$ on $\Ha(W)$ in the following way:
\begin{equation}\label{alphaac}
\alpha:\Ha(W) \times [\widetilde{F^{\times}}_+\Sp(W)] \longrightarrow \Ha(W); ((v,t), \widetilde{g}) \longmapsto (\sqrt{\widetilde{\lambda}_{ \widetilde{g}}}^{-1} v  \widetilde{g}, t).
\end{equation}
The restriction of $\alpha$ on $\Sp(W)$ is the usual action. By observation,  this action factors through the following group homomorphism:
$$ \widetilde{F^{\times}}_+\Sp(W)\simeq \widetilde{F^{\times}}_+ \times \Sp(W) \longrightarrow \Sp(W);(t, g) \longmapsto  g.$$
 Our next purpose is to extend this action to $\wideparen{\widetilde{\GSp}}(W)$. For any element $\widetilde{g}= (g,\epsilon)\in \widetilde{\GSp}(W)$, $\widetilde{\lambda}_{\widetilde{g}}=(\lambda_{g}, \epsilon)\in \widetilde{F^{\times}}$.
 Let us write
$$\lambda_g=a_g^2 t_g, $$
for $a_g^2\in F^{\times 2}$, and $t_g=\kappa(\dot{\lambda}_g)\in \{ 1, \zeta_1, \varpiup, \zeta_1\varpiup\}$. It corresponds to an element $([a_g^2,1], \dot{t}_g)$ of $\widetilde{F^{\times}}$. Then:
$$\widetilde{\lambda}_{\widetilde{g}}=[a_g^2,\epsilon][ t_g,1],$$
for $[a_g^2,\epsilon]\in \widetilde{F^{\times 2}}$. Let us define
$$ a_{\widetilde{g}}=\sqrt{[a_g^2,\epsilon]} \in \widetilde{F^{\times}}_+$$
\begin{lemma}
\begin{itemize}
\item[(1)] $a_{\widetilde{g}}^{-1}a_{\widetilde{g'}}^{-1}=\sqrt{c''(\widetilde{\lambda}_{\widetilde{g}}, \widetilde{\lambda}_{\widetilde{g'}})} a_{\widetilde{g}\widetilde{g'}}^{-1}=\wideparen{c''}(\widetilde{g},\widetilde{g'})a_{\widetilde{g}\widetilde{g'}}^{-1}$, for $\widetilde{g}=(g,\epsilon), \widetilde{g'}=(g',\epsilon')\in \widetilde{\GSp}(W)$.
\item[(2)] $\widetilde{\lambda}_{a_{\widetilde{g}}^{-1}\widetilde{g}}\widetilde{\lambda}_{a_{\widetilde{g'}}^{-1}\widetilde{g'}}=c''(\widetilde{g}, \widetilde{g'})\widetilde{\lambda}_{a_{\widetilde{g}\widetilde{g'}}^{-1}\widetilde{g}\widetilde{g'}}$, for $\widetilde{g}, \widetilde{g'}\in  \widetilde{\GSp}(W)$.
\end{itemize}
\end{lemma}
\begin{proof}
1) Let us write
$$\widetilde{\lambda}_{\widetilde{g}}=[a_g^2,\epsilon][ t_g,1], \widetilde{\lambda}_{\widetilde{g'}}=[a_{g'}^2,\epsilon'][ t_{g'},1],\widetilde{\lambda}_{\widetilde{g}\widetilde{g'}}=[a_{gg'}^2,\epsilon''][ t_{gg'},1],$$
for $t_g=\kappa(\dot{\lambda}_g)$, $t_{g'}=\kappa(\dot{\lambda}_{g'})$, $t_{gg'}=\kappa(\dot{\lambda}_{gg'})$.
As $\widetilde{\lambda}_{\widetilde{g}}\widetilde{\lambda}_{\widetilde{g'}}=\widetilde{\lambda}_{\widetilde{g}\widetilde{g'}}$, by (\ref{c33}),
$$[a_g^2,\epsilon][ t_g,1][a_{g'}^2,\epsilon'][ t_{g'},1]=[a_g^2a_{g'}^2, c'''(a_g^2, a_{g'}^2)\epsilon\epsilon'] [t_gt_{g'},  c'''(t_g, t_{g'})]$$
$$=[a_g^2a_{g'}^2, c_{\sqrt{\,}}(a_g^2, a_{g'}^2)\epsilon\epsilon'] [t_gt_{g'},  1]=[a_g^2a_{g'}^2, c_{\sqrt{\,}}(a_g^2, a_{g'}^2)\epsilon\epsilon'] [c'(\lambda_g,\lambda_{g'})t_{gg'},  1]$$
$$=[a_g^2a_{g'}^2, c_{\sqrt{\,}}(a_g^2, a_{g'}^2)\epsilon\epsilon'][c'(\lambda_g,\lambda_{g'}),1] [t_{gg'},  1]=[a_g^2a_{g'}^2, c_{\sqrt{\,}}(a_g^2, a_{g'}^2)\epsilon\epsilon']c''(\lambda_g,\lambda_{g'}) [t_{gg'},  1].$$
So $$[a_g^2,\epsilon][a_{g'}^2,\epsilon']c''(\lambda_g,\lambda_{g'})=[a_g^2,\epsilon][a_{g'}^2,\epsilon']c''(\widetilde{\lambda}_{\widetilde{g}}, \widetilde{\lambda}_{\widetilde{g'}})=[a_{gg'}^2,\epsilon''].$$
Hence the result holds.\\
2) It is a consequence of (1).
\end{proof}
Let us extend the twisted  action $\alpha$ to $\wideparen{\widetilde{\GSp}}(W)$ as follows:
\begin{equation}\label{alphaacmod411111}
\alpha: \Ha(W) \times \wideparen{\widetilde{\GSp}}(W) \longrightarrow \Ha(W); ((v,t), (\widetilde{g},k)) \longmapsto (a_{\widetilde{g}}^{-1} v \widetilde{g}k, \lambda_{a_{\widetilde{g}}^{-1}\widetilde{g}} \lambda_{k}  t).
\end{equation}
By the above lemma, it is well-defined. In this action,  $\widetilde{F^{\times}}_+$  acts trivially on $\Ha(W)$,  so it indeed defines an action of $\wideparen{\PGSp^{\pm}}(W)$ on $\Ha(W)$.

\subsection{Extended to $\PGMp^{\pm}(W)$}
 Via the projections:  $\PGMp^{\pm}(W) \to \PGSp^{\pm}(W)$ and  $\overline{\PGSp^{\pm}}(W) \to \PGSp^{\pm}(W)$, we  view $\wideparen{c''(-,-)}$ as a cocycle on $\PGMp^{\pm}(W)$ and $\overline{\PGSp^{\pm}}(W)$.  Let   $\wideparen{\PGMp^{\pm}}(W)$,  $\wideparen{\overline{\PGSp^{\pm}}}(W)$, denote the corresponding covering groups associated to $\wideparen{c''(-,-)}$.  Then there exists the following commutative diagrams:
\begin{equation}\label{eq15}
\begin{CD}
1@>>>\overline{\Sp}(W) @>>>\wideparen{\overline{\PGSp^{\pm}}}(W)@> \wideparen{\dot{\lambda}}>> \wideparen{F^{\times}/F^{\times 2}} @>>> 1\\
@. @VVV @VV\wideparen{p}V @\vert @.\\
1 @>>> \Sp(W) @>>> \wideparen{\PGSp^{\pm}}(W)@>\wideparen{\dot{\lambda}}>> \wideparen{F^{\times}/F^{\times 2}} @>>> 1\\
\end{CD}
\end{equation}
\begin{equation}\label{eq16}
\begin{CD}
1@>>>\Mp(W) @>>>\wideparen{\PGMp^{\pm}}(W)@> \wideparen{\dot{\lambda}}>> \wideparen{F^{\times}/F^{\times 2}} @>>> 1\\
@. @VVV @VV\wideparen{p}V @\vert @.\\
1 @>>> \Sp(W) @>>> \wideparen{\PGSp^{\pm}}(W)@>\wideparen{\dot{\lambda}}>> \wideparen{F^{\times}/F^{\times 2}} @>>> 1\\
\end{CD}
\end{equation} 
\subsection{Extended  Weil representations}
Recall that $(\pi_{\psi}, V_{\psi})$ is a Weil representation of $\Mp(W)$ associated to $\psi$.  In this case,  we define the twisted induced Weil representation of $\PGMp^{\pm}(W)$ as follows:
$$\Pi_{\psi}=\cInd_{\Mp(W)}^{\PGMp^{\pm}(W)} \pi_{\psi},\quad \mathcal{V}_{\psi}=\cInd_{\Mp(W)}^{\PGMp^{\pm}(W)} V_{\psi}.$$
By considering the restriction of $\Pi_{\psi}$ on $\Mp(W)$, we can obtain four Weil representations, corresponding  to $\pi_{\psi^a}$, as $a=1, \zeta_1,  \varpiup,\zeta_1\varpiup$. So it is independent  of  the choice of $\psi$. By (\ref{zzcoc114122}), this representation can also be defined over $\overline{\PGSp^{\pm}}(W)$. By Clifford-Mackey theory, it is also compatible with the induced Weil representation directly from $\overline{\Sp}(W)$ to $\overline{\PGSp^{\pm}}(W)$.  The next purpose is to see whether such  representations  can arise    from the Heisenberg  representation of $\Ha(W)$.

Recall the definition of the $\alpha$ action  in (\ref{alphaac}).  Through the  projection $\wideparen{p}$, we  extend this  action to $\wideparen{\overline{\PGSp^{\pm}}}(W)$ and $\wideparen{\PGMp^{\pm}}(W)$. Then let us define a representation as follows:
$$\wideparen{\Pi}_{\psi}=\cInd_{\Mp(W)\ltimes \Ha(W)}^{\wideparen{\PGMp^{\pm}(W)} \ltimes_{\alpha} \Ha(W)}\pi_{\psi},\quad \wideparen{\mathcal{V}}_{\psi}=\cInd_{\Mp(W)\ltimes \Ha(W)}^{\wideparen{\PGMp^{\pm}(W)} \ltimes_{\alpha} \Ha(W)} V_{\psi}.$$
Notice that $\wideparen{\PGMp^{\pm}(W)}$ is a covering group over $\PGMp^{\pm}(W)$. So there also  exists a  question on how to reasonably eliminate the effects of the central group.  Similarly, we can define the corresponding representation of $\wideparen{\overline{\PGSp^{\pm}}}(W)\ltimes_{\alpha} \Ha(W)$.

\labelwidth=4em
\addtolength\leftskip{25pt}
\setlength\labelsep{0pt}
\addtolength\parskip{\smallskipamount}


\begin{thebibliography}{99}



\bibitem{Ba} L.Barthel,
 {\it Local Howe correspondence for groups of similitudes},
  J. Reine Angew. Math., 414 (1991), 207-220.

\bibitem{GaGrPr} W. T.Gan, B. H. Gross and D. Prasad,
 {\it Twisted GGP Problems and Conjectures}, Compositio Mathematica, Volume 159, Issue 9, September 2023, pp. 1916-1973.


\bibitem{GePi} S.Gelbart, I.Piatetski-Shapiro,
{\it  Some remarks on metaplectic cusp forms and the correspondences of Shimura and Waldspurger}, Israel J. Math. 44(1983), 97-126.

\bibitem{Ge} L.Gerstein,
{\it Basic Quadratic Forms}, Graduate Studies in Mathematics, Vol. 90, AMS, Providence, 2008.

\bibitem{Ku} S.S.Kudla,
 {\it  Notes on the local theta correspondence},
 preprint, available at http://www.math.utotonto.ca/~skudla/castle.pdf, 1996.

\bibitem{Ll} J. L. Lahorgue,
{\it Groups of order $8$ and $16$}, Master. France. 2018. cel-01841041.

\bibitem{MoViWa} C.M\oe glin,  M.-F.Vign\'eras,  J.-L.Waldspurger,
{\it Correspondances de Howe sur un corps $p$-adique},
Lect. Notes Math. 1291, Springer, 1987.

\bibitem{Pa} S.P.Patel,
{\it Branching laws on the metaplectic
cover of $\GL_2$}, thesis, 2014.


\bibitem{Pe} P.Perrin,
{\it Repr\'esentations de Schr\" odinger. Indice de Maslov et groupe metaplectique},
in ``Non commutative Harmonic Analysis and Lie Groups'', Lect. Notes Math. 880 (1981), 370-407.

\bibitem{Ra} R.R.Rao,
{\it On some explicit formulas in the theory of the Weil representation},
Pacific J. Math. 157 (1993), 335-371.

\bibitem{Wa1} J.-L. Waldspurger,
{\it Correspondance de Shimura},
J. Math. Pures Appl. (9) 59 (1980), no. 1, 1-132.

\bibitem{Wa2} J.-L. Waldspurger,
{\it Correspondances de Shimura et quaternions},
Forum Math. 3 (1991), no. 3, 219-307.


\bibitem{Wa3} C.-H. Wang,
{\it Extended Weil representations by some twisted actions: the finite field cases},
arXiv:2204.03987.


\bibitem{We} A.Weil,
{\it  Sur certains groupes d'op\'erateurs unitaires},
Acta Math. 111 (1964), 143-211.

\end{thebibliography}
\end{document}